\documentclass[a4paper,12pt,reqno]{amsart}

\usepackage[english]{babel}
\usepackage[latin1]{inputenc}

\usepackage[verbose,tmargin=1in,bmargin=1in,lmargin=1in,rmargin=1in]{geometry}
\usepackage{amsfonts}
\usepackage{amsmath}
\usepackage{amssymb}
\usepackage{amsthm}
\usepackage{graphicx}
\usepackage{caption}
\usepackage{color}
\usepackage[unicode=true,pdfusetitle,
    bookmarks=true,bookmarksnumbered=false,bookmarksopen=false,
breaklinks=true,pdfborder={0 0 0},backref=false,colorlinks=true]{hyperref}
\hypersetup{linkcolor=blue,citecolor=blue,urlcolor=blue}

\newtheorem{theorem}{Theorem}[section]
\newtheorem{remark}{Remark}[section]
\newtheorem{proposition}{Proposition}[section]
\newtheorem{lemma}{Lemma}[section]

\begin{document}

\author{D.-A.\ Deckert}
\email{deckert@math.ucdavis.edu}
\address{Mathematisches Institut der LMU, Theresienstr. 39, 80333
M\"unchen, Germany.}

\author{G.\ Hinrichs}
\email{hinrichs@math.lmu.de}
\address{Mathematisches Institut der LMU, Theresienstr. 39, 80333
M\"unchen, Germany.}

\title[Electrodynamic Two-Body Problem on the Straight Line]{Electrodynamic
Two-Body Problem\\for prescribed Initial Data\\on the Straight line}

\begin{abstract}
    Due to the finite speed of light, direct electrodynamic interaction between
    point charges can naturally be described by a system of ordinary differential
    equations involving delays. As electrodynamics is time-symmetric, these
    delays appear as time-like retarded as well as advanced arguments in the
    fundamental equations of motion -- the so-called
    Fokker-Schwarzschild-Tetrode (FST) equations. However, for special initial
    conditions breaking the time-symmetry, effective equations can be derived
    which are purely retarded. Dropping radiation terms, which in many
    situations are very small, the latter equations are called Synge equations.
    In both cases, few mathematical results are available on existence of
    solutions, and even fewer on uniqueness. We investigate the situation of two
    like point-charges in $3+1$ space-time dimensions restricted to motion on a straight
    line.  We give a priori estimates on the asymptotic motion and, using a
    Leray-Schauder argument, prove: 1) Existence of solutions to the FST
    equations on the future or past half-line given finite trajectory strips; 2) Global
    existence of the Synge equations for Newtonian Cauchy data; 3) Global
    existence of a FST toy model that involves advanced and retarded terms.
    Furthermore, we give a sufficient criterion that uniquely distinguishes
    solutions by means of finite trajectory strips.
\end{abstract}

\maketitle

\section{Introduction and Main Results}

The direct electrodynamic interaction between point-charges is a prime example
for systems of ordinary differential equations involving delays. Its fundamental
equations of motion can be inferred by means of an informal variational
principle of the action $S$ which is given as a functional of the world-lines of the
point-charges, $z_i:\mathbb R\to\mathbb R^4$, $\tau\mapsto z_i(\tau)$: 
\begin{equation}
\label{eq:action}
\begin{split}
    S[z_{i=1,\ldots,n}]
    =&
    -\sum_{i=1}^N \int m_i\sqrt{dz_{i\mu}(\tau)dz_i^\mu(\tau)} \\
    &-\sum_{i=1}^N\sum_{j=1,j\neq i}^N \frac{e_ie_j}{2} \int dz_{i\mu} \int
    dz_j^\mu\delta\Big((z_i(\tau)-z_j(\sigma))^2\Big)\,.
    \end{split}
\end{equation}
In this relativistic notation, $\tau$ is the parametrization of the world-line,
and $z_i=(z_i^\mu)_{\mu=0,1,2,3}$ denotes the time and space  coordinates, $z_i^0$
and $\mathbf z=(z_i^j)_{j=1,2,3}$, respectively. The integral $\int dz_i^\mu$ is
to be interpreted as the line integral $\int d\tau \dot z_i\mu(\tau)$, where
dots denote derivatives w.r.t.\ $\tau$. Furthermore, the summation convention
$a_\mu b^\mu=a^0 b^0 - \mathbf a\cdot \mathbf b$ is used so that $a^2=a_\mu
a^\mu$ equals the square of the relevant, indefinite Minkowski metric in
relativistic space-time. The symbol $\delta$ denotes the one-dimensional Dirac
delta distribution, $m_i$ denotes the mass of the particles, $e_i$ the respective
charge, and we chose units such that the speed of light and the electric
constant equal one. The integral
in the first summand in \eqref{eq:action} measures the arc length of the $i$-th
world line using the Minkowski metric, and the double integral in the second summand gives rise to an
interaction between pairs of world lines whenever the Minkowski distance between
$z_i$ and $z_j$ is zero. The extrema of the action $S$, i.e., $z_i$ such that
$\left.\frac{\text d}{\text d\epsilon}\right|_{\epsilon=0}S[z_i+\epsilon\delta z_i]=0$, fulfill the FST equations (also
known as Wheeler-Feynman equations):
\begin{align}
    \label{eq:Lorentz-force}
    m_i \ddot z_i^\mu(\tau) 
    = 
    e_i \sum_{j=1,j\neq i}^N 
        \frac12 \sum_\pm 
        F^{\mu\nu}_{j\pm}\big(z_i(\tau)\big)\dot z_{i\nu}(\tau),
        \qquad
        i=1,2,\ldots,N
\end{align}
with the electromagnetic field tensors $F_{j\pm}^{\mu\nu}(x)=\partial/\partial x_\mu A_j^\nu(x)-\partial/\partial x_\nu
A^\mu_j(x)$ given by means of the four-vector potentials
\begin{align}
    \label{eq:potentials-and-delay}
    A_{j\pm}^\mu(x) 
    = 
    e_j \frac{\dot z_{j\pm}}{(x-z_{j\pm})_\mu \dot z_{j\pm}^\mu},
    \qquad
    z_{j\pm}^\mu = z_j^\mu(\tau_{j\pm}),
    \qquad
x^0-z_j^0(\tau_{j\pm})=\pm |\mathbf x - \mathbf z_j(\tau_{j\pm})|.
\end{align}
Equation \eqref{eq:Lorentz-force} is the special relativistic form of Newton's
force law, in electrodynamics referred to as Lorentz equation. The field tensors
$F^{\mu\nu}_{j+}$, $F^{\mu\nu}_{j-}$ are the so-called advanced ($+$) and retarded ($-$)
electrodynamic Liénard-Wiechert fields \cite{rohrlich} which are generated by the $j$-th
charge, respectively. They are given in terms of the corresponding potentials $A^\mu_{j+}$, $A^\mu_{j-}$, 
which are functionals of the world line $\tau\mapsto z_j(\tau)$ since the
parameters $\tau_{j\pm}$ are defined implicitly as solutions to the last equation
in line \eqref{eq:potentials-and-delay}. This implicit equation is due to the
delta function in \eqref{eq:action} and has a nice geometrical interpretation.
When evaluating $F_{j\pm}^{\mu\nu}(x)$ at $x=z_i$ as in the Lorentz force
\eqref{eq:Lorentz-force}, the respective $\tau_{j\pm}$ identify the space-time
points $z_{j\pm}$ which can be reached with speed of light. The existence of
both $\tau_{j+}$ and $\tau_{j-}$ is ensured as long as the world lines have
velocities smaller than speed of light. Their values are however not bounded
a priori. Since computing $F^{\mu\nu}_{j\pm}(x)$ involves taking another
derivative of $A^\mu_{j\pm}(x)$ w.r.t. to $x$, and $\tau_{j\pm}$ depends on $x$,
the right-hand side of the Lorentz equations \eqref{eq:Lorentz-force} involves
advanced and retarded four-vectors $z_{j\pm}^\mu$, four-velocities $\dot
z_{j\pm}^\mu$, and four-accelerations $\ddot z_{j\pm}^\mu$ -- hence,
\eqref{eq:Lorentz-force} is a neutral equation of mixed-type, state-dependent,
and has unbounded delay. \\

The formulation of electrodynamics by means of direct interaction as in
\eqref{eq:Lorentz-force} is due to ideas and works of Gauss \cite{Gauss}, Fokker,
\cite{Fokker} Tetrode \cite{Tetrode}, Schwarzschild \cite{Schwarzschild}.
Wheeler and Feynman \cite{WF1,WF2}
showed that this formulation is capable of explaining the irreversible nature of
radiation. Beyond that, it is the only candidate for a singularity free
formulation of classical electrodynamics. For a more detailed discussion, see
also the overview article \cite{GernotDirkDetlefGuenter}. Mathematically, however, even global
existence of solutions to \eqref{eq:Lorentz-force} is an open problem.  The
few rigorous results available apply to special situations only. In the special case of two
like charges restricted on the straight line, global existence for prescribed
asymptotic data was proven in \cite{Bauer}.  In the case of
$N$ arbitrary extended charges in $3+1$ space-time dimensions, the
existence of solutions for prescribed Newtonian Cauchy data on finite, but
arbitrary large time intervals was shown in \cite{Dirk2}. If, furthermore, the two charges are
initially sufficiently far apart and have zero velocities the corresponding
solutions were shown to be unique \cite{Driver2}.\\

In this paper we also study the case of two like charges and restrict us to
solutions that describe motion on a straight line as in \cite{Bauer}. However, in
contrast to \cite{Bauer}, we aim at existence of solutions for prescribed data at
finite times instead of asymptotic data. As long as the motion takes place on a
straight line, the equations of motion \eqref{eq:Lorentz-force} can be simplified
as follows. To keep the notation short, we express \eqref{eq:Lorentz-force} in
coordinates and introduce the trajectories $a,b$ as maps $\mathbb R\to \mathbb
R$, $t\mapsto a(t)$ and $t\mapsto b(t)$, such that $z_1=(t,a(t),0,0)$ and
$z_2=(t,b(t),0,0)$. Equations
\eqref{eq:Lorentz-force}-\eqref{eq:potentials-and-delay} then turn
into
\begin{equation}\label{WF}
    \begin{split}
        &\frac{\text d}{\text dt}\left(\frac{\dot a(t)}{\sqrt{1-\dot a(t)^2}}\right) = 
        \kappa_a\left[\epsilon_-\frac{\rho(b\left(t_2^-\right)}{\left(a(t)-b(t_2^-)\right)^2} 
        + \epsilon_+ \frac{\sigma(\dot b\left(t_2^+\right))}{\left(a(t)-b(t_2^+)\right)^2}\right] \\
        &\frac{\text d}{\text dt}\left(\frac{\dot b(t)}{\sqrt{1-\dot b(t)^2}}\right) = 
        -\kappa_b\left[\epsilon_-
            \frac{\sigma(\dot a\left(t_1^-\right))}{\left(b(t)-a(t_1^-)\right)^2}
        +\epsilon_+\frac{\rho(\dot a\left(t_1^+\right))}{\left(b(t)-a(t_1^+)\right)^2}\right]
    \end{split}
\end{equation}
for
\begin{align}
    \label{eq:vel-factors}
    \rho(v)=\frac{1+v}{1-v}\,,
    \qquad
    \sigma(v)=\frac{1-v}{1+v}\,.
\end{align}
Here, we introduced the following parameters for convenience of our discussion:
$\kappa_{a/b} :=\frac{e_a e_b}{m_{a/b}}$
which is the coupling constant,  
and $\epsilon_+$ and $\epsilon_-$ which allow to individually switch the advanced or retarded
terms on and off. The time-symmetric FST equations 
\eqref{eq:Lorentz-force} are recovered by setting 
$\epsilon_+=\frac12=\epsilon_-$. The analogs of the parameters
$\tau_{i\pm}$ given in \eqref{eq:potentials-and-delay}, expressed
in coordinates, are the advanced and retarded times
$t_1^\pm(a,b,t)$, $t_2^\pm(a,b,t)$. These times are functions of the
trajectories $a,b$ and time
$t$ defined by
\begin{equation}\label{t+-}
        t_1^\pm(a,b,t)=t\pm|a(t_1^\pm(a,b,t))-b(t)|\,, \qquad
        t_2^\pm(a,b,t)=t\pm|a(t)-b(t_2^\pm(a,b,t))|\,.
\end{equation}
Their dependence on $a,b,t$ will
often be suppressed in the notation. Beside $\epsilon_\pm=\frac12$, another
interesting case of \eqref{WF} is given by $\epsilon_+=0$ and $\epsilon_-=1$,
which results in the so-called Synge equations. It was shown by Wheeler and
Feynman that solutions to the time-symmetric fundamental equations of motion
\eqref{eq:Lorentz-force}, in the case of special initial configurations that
break the time symmetry, are effectively also solutions to equations of motion
involving a radiation reaction term and only retarded delays. Upon neglecting radiation reaction terms, which
for small charges (like the electron charge) and small accelerations give only
small corrections, these approximate equations take the form of the Synge
equations. For them, the following mathematical results are
available: In $3+1$ space-time dimensions existence of solutions 
for times $t\geq0$ and given admissible trajectory histories was discussed
in \cite{Angelov}. In the case of $N$ extended charges, existence and uniqueness of
solutions to the Synge equations for prescribed histories was shown in \cite{Dirk2}.
On the straight line, uniqueness w.r.t. Newtonian Cauchy data for like charges
that are initially sufficiently far apart and have zero velocities was shown in \cite{Driver1}. \\

Our present work extends these results as follows.  For the case of the FST
equations, i.e., $\epsilon_-=\frac12=\epsilon_+$, we prove existence of
solutions on the future or past half-line for initial data that consist of position and velocity of charge $a$
and a trajectory strip of charge $b$ whose ends are intersection points of the
light-cone through $\left(0,a(0)\right)$, see
Figure~\ref{fig:WF-Anfangsdaten}(i):
\begin{figure}
    \includegraphics[trim=2.3cm 14cm 2.8cm 0.6cm, clip, width=.8\textwidth]{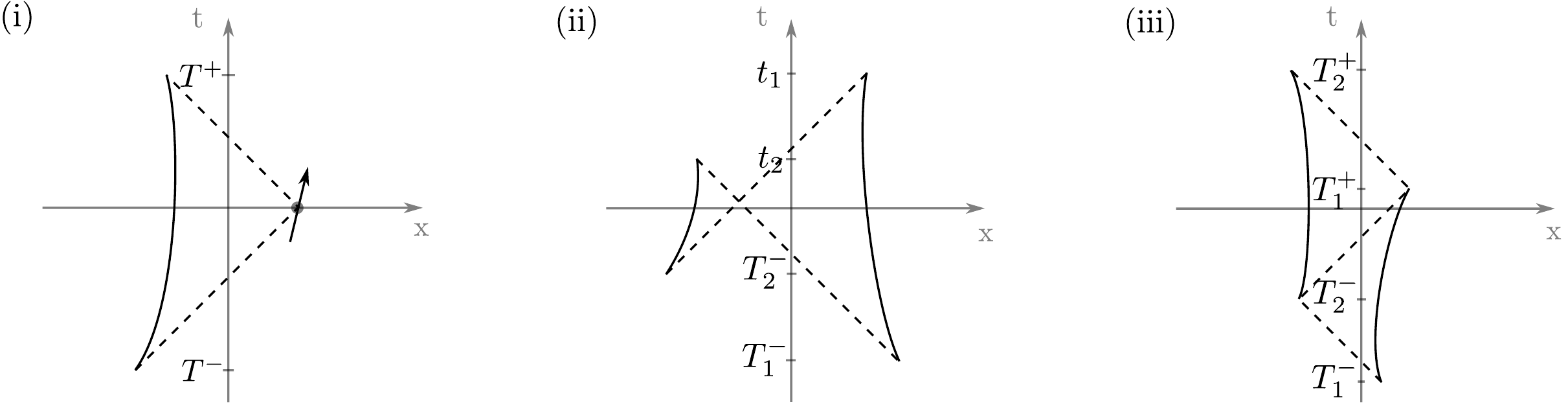}
    \caption[...]{\label{fig:WF-Anfangsdaten}Initial data for the FST equations required for
     (i) the existence result in Theorem~\ref{thm:+-};
     (ii) the alternative existence result mentioned in remark \ref{rem:Cauchy};
     (iii) the unique reconstruction of solutions as in theorem
     \ref{thm:Rekonstruktion}.}
\end{figure}
\begin{theorem}[Fokker-Schwarzschild-Tetrode Equations]\label{thm:+-}
    Let $\epsilon_+=\frac12=\epsilon_-$.
    Given initial position and velocity 
    \begin{align}
        \label{eq:initial-a}
    a_0\in\mathbb R,\, \dot a_0\in]-1,1[
    \end{align}
    of charge $a$ and an initial
    trajectory strip 
    \begin{align}
        \label{eq:initial-b}
       b_0\in C^1\left([T^-,T^+],]-\infty,a_0[\right) 
    \end{align}
    of charge $b$
    with $T^{\pm}=\pm(a_0-b_0(T^{\pm}))$ and $\|\dot
    b_0\|_\infty<1$, the following holds:
    \begin{enumerate}
        \item[a)] There is at least one pair of trajectories
            \begin{align}
                a\in C^2(\mathbb
                R_0^+)\text{ and }b\in
                C^1([T^-,\infty[)\cap C^2(]T^+,\infty)
            \end{align}such that
            the first equation of \eqref{WF} together with
            \eqref{eq:vel-factors}-\eqref{t+-} is
            satisfied for all $t\ge0$ and the second one of \eqref{WF} holds for all $t\ge T^+$, and
            furthermore
            \begin{equation}\label{Anfangswerte}
                a(0)=a_0,\,\dot a(0)=\dot a_0,\, \qquad \left.b\right|_{[T^-,T^+]}=b_0\,.
            \end{equation}
        \item[b)] If $b_0\in C^{1+n}$, $n\in\mathbb N_0$, then the
            regularity of any solution is characterized as follows: Let
            $\sigma_0:=0, \tau_0:=T^+$, $\sigma_1:=t_1^+(T^+),
            \tau_1:=t_2^+(\sigma_1)$, and $\sigma_{k+1}:=t_1^+(\tau_k),
            \tau_{k+1}:=t_2^+(\sigma_{k+1})$, then $a$ is $2k$ times
            differentiable at $\sigma_k$ for $k\in\mathbb N$, $b$ is $2k+1$
        times differentiable at $\tau_k$ for $k\in\mathbb N_0$, and
    $\left.a\right|_{]\sigma_k, \sigma_{k+1}[}\in C^{2+n+2k}$ and
    $\left.b\right|_{]\tau_k, \tau_{k+1}[}\in C^{3+n+2k}$ for any $k\in\mathbb
        N_0$; see Figure~\ref{fig:Glattheit}.
    \end{enumerate}
\end{theorem}
\begin{figure}
    \includegraphics[trim=0cm 5.4cm 0cm 0.5cm, clip, width=0.3\textwidth]{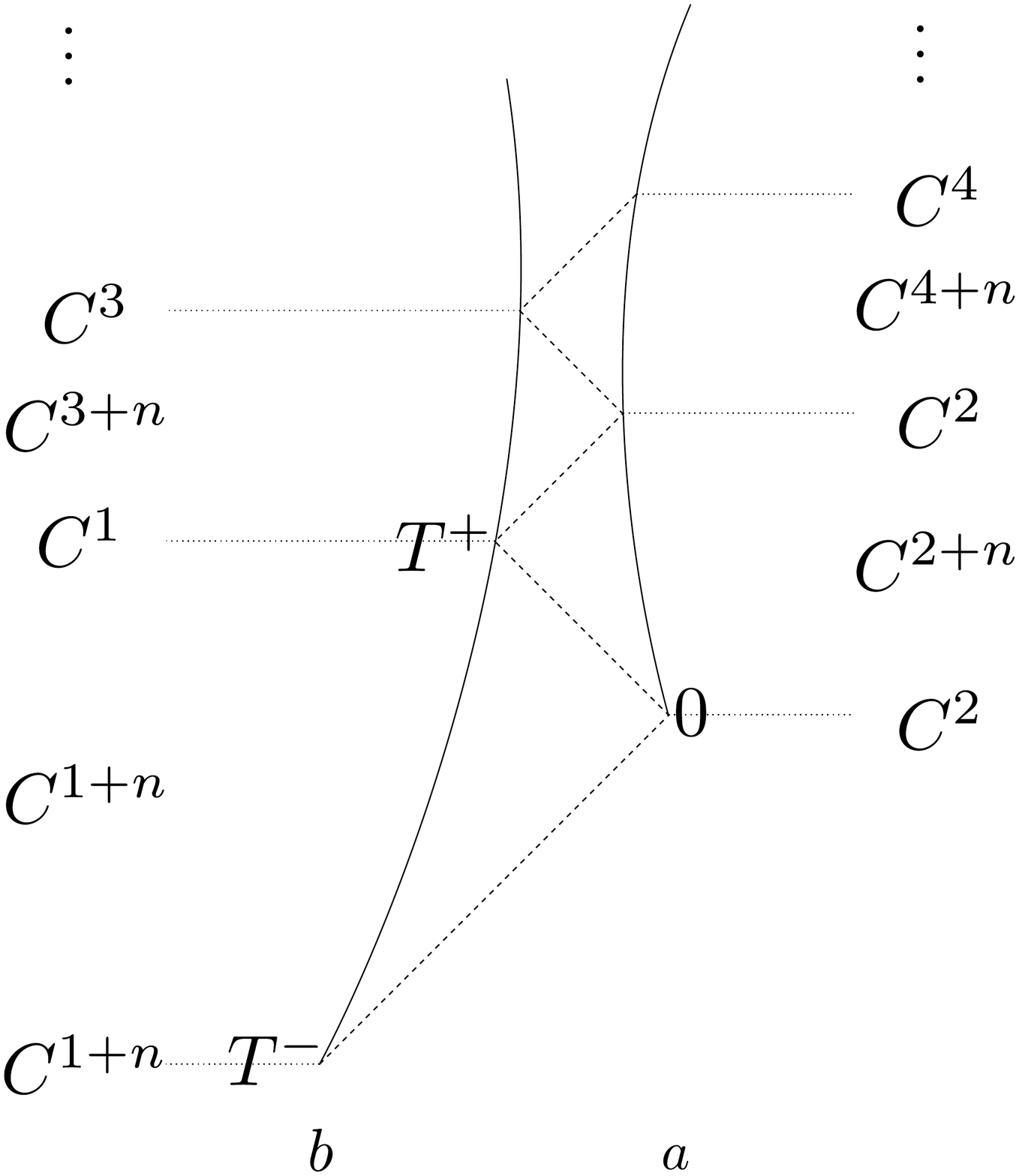}
    \caption[...]{\label{fig:Glattheit}Smoothness of solutions to \eqref{WF}
    with initial data \eqref{eq:initial-a}, \eqref{eq:initial-b} such that $b_0\in
    C^{1+n}$.
     }
\end{figure}
Here, the set $C^n(D,E)$) refers to $n$ times continuously differential
functions on $D$ with values in $E$ where derivatives at boundary points are to
be understood as one-sided ones.  By a similar argument, one can also show
existence of solutions $(a,b)$ in the past of $a(0)$ and $b(T^-)$.  The proof is
based on an a priori estimate on the asymptotic behavior of FST
solutions provided by Proposition~\ref{prop:Asymptotik+-} below. The trajectory
data needed to establish the a priori estimate is exactly the required data
\eqref{eq:initial-a}-\eqref{eq:initial-b}. This estimate allows an application
of Leray-Schauder's fixed-point theorem. \\

For the Synge equations, i.e., \eqref{WF}-\eqref{t+-} for $\epsilon_+=0$ and
$\epsilon_-=1$, we are able to control the asymptotic behavior of solutions
knowing Newtonian Cauchy data only, i.e., positions and velocities at
one time instant. In turn, this enables us to show global existence of solutions
to the Synge equations for any given Newtonian Cauchy data. 
\begin{theorem}\label{thm:-}
 Let $\epsilon_-=1,\,\epsilon_+=0$.
 For any Newtonian Cauchy data $a_0>b_0$ and $\dot a_0,\,\dot b_0\in]-1,1[$,
     there is at least one pair of trajectories $a,b\in
     C^\infty(\mathbb R)$ that solves \eqref{WF}-\eqref{t+-} with
 \begin{equation}\label{Anfangswerte-}
 a(0)=a_0,\, \dot a(0)=\dot a_0,\, \qquad b(0)=b_0,\, \dot b(0)=\dot b_0\,.
 \end{equation}
\end{theorem}
With the same technique one can also treat the case
$\epsilon_-=0,\,\epsilon_+=1$.  We note that, although we only regard motion
along the straight line, this global existence results goes somewhat beyond the
existence results given in \cite{Driver1, Angelov, Dirk2}, which treat existence
of solutions only for the half-line $t\geq 0$ given prescribed histories.
\begin{remark}\label{rem:Cauchy}
    Due to the Lorentz invariance of the Synge equations it is also possible to
    specify the ``initial'' positions and velocities at different times $t_1$
    and $t_2$, provided the space-time points $(t_1,a(t_1))$ and $(t_2,b(t_2))$
    are space-like separated. Alternative initial data for the
    time-symmetric equations are obtained by choosing $(t_1,a(t_1))$ and
    $(t_2,a(t_2))$ space-like separated and prescribing both trajectories up to
    the corresponding retarded times, i.e., specifying
    $\left.a\right|_{[T_1^-,t_1]}$ and $\left.b\right|_{[T_2^-,t_2]}$ such that
    $T_1^-=t_2-\left[a(T_1^-)-b(t_2)\right]$ and
    $T_2^-=t_1-\left[a(t_1)-b(T_2^-)\right]$; see
    Figure~\ref{fig:WF-Anfangsdaten}(ii).
\end{remark}
One may wonder if the reason why global existence of solution to the Synge
equations can be shown already for any Newtonian Cauchy data instead of giving
initial strips of the trajectories stems from the fact that only retarded delays
are involved and advanced delays are absent (or vice versa).  This is however
not the case. To see this, we also regard a
FST toy model inferred from \eqref{WF}-\eqref{t+-} by setting
$\epsilon_+=\frac12=\epsilon_-$ and $\sigma=1=\rho$ which results in:
\begin{equation}\label{Spielmodell}
 \begin{split}
 &\frac{\text d}{\text dt}\left(\frac{\dot a(t)}{\sqrt{1-\dot a(t)^2}}\right) = 
 \frac1{\left(a(t)-b(t_2^-)\right)^2} +\frac1{\left(a(t)-b(t_2^+)\right)^2} \\
 &\frac{\text d}{\text dt}\left(\frac{\dot b(t)}{\sqrt{1-\dot b(t)^2}}\right) = 
 -\frac1{\left(b(t)-a(t_1^-)\right)^2} -\frac1{\left(b(t)-a(t_1^+)\right)^2}\,;
 \end{split}
\end{equation}
see \cite{DirkNicola} for a more detailed discussion of this model in 3+1
space-time dimensions.
Although the resulting equation of motion \eqref{Spielmodell} involves advanced as well as retarded delays,
we were able to prove global existence of solutions for Newtonian Cauchy data
with the same technique. In conclusion, the technical obstacle to infer a
similar result for \eqref{WF} is due to the lack of control of the
denominators of the velocity
factors $\sigma$ and $\tau$ in \eqref{eq:vel-factors} by means of Newtonian
Cauchy data only, and not necessarily due to the
time-symmetry.
\begin{theorem}\label{thm:Spielmodell}
 For any Newtonian Cauchy data $a_0>b_0$ and $\dot a_0,\,\dot b_0\in]-1,1[$,
     there is at least one pair of trajectories $a,b\in
     C^\infty(\mathbb R)$ that solves \eqref{Spielmodell} and \eqref{t+-} with
      \begin{equation}\label{Anfangswerte-spiel}
 a(0)=a_0,\, \dot a(0)=\dot a_0,\, b(0)=b_0,\, \dot b(0)=\dot b_0\,.
 \end{equation}
\end{theorem}

These existence results do not touch upon the question of uniqueness. In the
case of the toy model \eqref{Spielmodell} it was shown in \cite{DirkNicola} that
at least for finite times solutions to \eqref{Spielmodell} and \eqref{t+-} can
be constructed by a what is commonly called a ``method of steps'' in the field
of delay differential equations. Here, we give necessary conditions to identify
solutions to \eqref{WF}-\eqref{t+-} uniquely:
\begin{theorem}\label{thm:Rekonstruktion}
    Let $\epsilon_+=\frac12=\epsilon_-$. 
    Any solution $(a,b)$ to equations \eqref{WF}-\eqref{t+-} can be uniquely
    reconstructed knowing only the trajectory strips
    \begin{align}
        \left.a\right|_{[T_1^-,T_1^+]}, \qquad \left.b\right|_{[T_2^-,T_2^+]}\,,
    \end{align}
    with times $T_1^+\in\mathbb R$, $T_2^+=t_2^+(T_1^+)$, $T_2^-=t_2^-(T_1^+)$ and
$T_1^-=t_1^-(T_2^-)$.
\end{theorem}
Similar results are possible for other combinations of $\epsilon_\pm$. In the
following sections we provide the proofs of the presented theorems.

\section{Proof of Theorem~\ref{thm:+-}}

Similarly to \cite{Bauer}, our proof is based on the following version of
Leray-Schauder's fixed point theorem:

\begin{theorem}[Leray-Schauder Theorem \cite{Granas}]\label{thm:Schauder}
Let ($\mathcal B,\|\cdot\|$) be a Banach space, $\mathcal O\subset\mathcal B$ a bounded open subset containing the origin and $\mathcal H:[0,1]\times \overline\Omega\rightarrow B$ a compact homotopy
such that $\mathcal H(0,\cdot)$ is the zero mapping and none of the mappings $\mathcal H(\lambda,\cdot)$ for $\lambda\in[0,1]$ has a fixed point on $\partial\mathcal O$.
Then $\mathcal H(1,\cdot)$ has a fixed point.
\end{theorem}

In order to apply the theorem, we need a Banach space consisting of pairs of
trajectories. For $(x,y)\in C^2(\mathbb R_0^+)\times \left\{\, C^1([T^-,\infty[)
    \, \cap \,
C^2(]T^+,\infty[)\,\right\}$, we define
\begin{equation}\label{Norm}
\|(x,y)\|:=\max\left(\|\dot x\|_\infty, \|\dot y\|_\infty, \sup_{t\ge0}|(1+|t|)\ddot x(t)|, \sup_{t>T^+}|(1+|t|)\ddot y(t)|\right) \,,
\end{equation}
choose
\begin{equation}\label{B}
    \begin{split}
        &B:=\bigg\{(x,y)\in C^2(\mathbb R_0^+)\times \left\{\,C^1([T^-,\infty[)\,\cap\,
        C^2(]T^+,\infty[)\, \right\} \bigg| \\
            &\hskip5cm x(0)=\dot x(0)=0,\,
            \left.y\right|_{[T^-,T^+]}\equiv0, \|(x,y)\|<\infty\bigg\}
        \end{split}
\end{equation}
for the Banach space, and equip it with the norm $\|\cdot\|$. Its particular choice will become clear below. From now on, we consider the values of the initial data $a_0$ and $b_0$ as fixed. Since nonzero initial data are not compatible with linearity,
they cannot directly be included in the definition of the Banach space $B$.
Instead, we fix a pair of reference trajectories
$(x_0, y_0)\in C^2(\mathbb R_0^+)\times \left\{\,C^1([T^-,\infty[)\,\cap\,
C^2(]T^+,\infty[)\,\right\}$
that satisfy the initial data and scatter apart without reaching too large
velocities or accelerations or too small distances. Precisely, we require
\begin{equation}\label{x0y0}
 \begin{split}
  &\|\dot x_0\|,\|\dot y_0\|<1,\\
  &\inf_{t\ge0}(x_0(t)-y_0(t))>0,\\
  &(\dot x_0-\dot y_0)(T^++1)>0,\\
  &\ddot x_0(t)\ge0\text{ for }t\ge0,\,\ddot y_0(t)\le0\text{ for }t> T^+\\
  &\ddot x_0(t)=\ddot y_0(t)=0\text{ for }t\ge T^++1\,.
 \end{split}
\end{equation}
The actual trajectories shall then given by
\begin{equation}\label{XY}
X:=x+x_0\,,\,\,Y:=y+y_0\,.
\end{equation}
Moreover, on a suitable subset of $B$, we define
the homotopy $H=(H_1,H_2)$ with range in $B$ and components
\begin{equation}\label{H}
 \begin{split}
 &H_1(\lambda,x,y)(t):=-x_0(t)+a_0+\int_0^t\frac{P_1(\lambda,X,Y)(s)}{\sqrt{1+P_1(\lambda,X,Y)(s)^2}}\text ds\,, \\
 &H_2(\lambda,x,y)(t):=\begin{cases}
                        0 & \text{ for } T^-\le t\le T^+ \\
			-y_0(t)+b_0(T^+)+\int_{T^+}^t\frac{P_2(\lambda,X,Y)(s)}{\sqrt{1+P_2(\lambda,X,Y)(s)^2}}\text ds & \text{ for }t\ge T^+\,.
			\end{cases}
 \end{split}
\end{equation}
Here, we used the abbreviations
\begin{equation}\label{P}
\begin{split}
 &P_1(\lambda,X,Y)(t):=(1-\lambda)\frac{\dot x_0(t)}{\sqrt{1-\dot x_0(t)^2}}+\lambda\frac{\dot a_0}{\sqrt{1-\dot a_0^2}}+\lambda\int_0^t\frac{\kappa_a}2F_1(X,Y)(s)\text ds\,, \\
 &P_2(\lambda,X,Y)(t):=(1-\lambda)\frac{\dot y_0(t)}{\sqrt{1-\dot y_0(t)^2}}+\lambda\frac{\dot b_0(T^+)}{\sqrt{1-\dot b_0(T^+)^2}}+\lambda\int_{T^+}^t\frac{\kappa_b}2F_2(X,Y)(s)\text ds\,,
 \end{split}
\end{equation}
and
\begin{equation}\label{F+-}
\begin{split}
 F_1(X,Y)(t):=&\frac{1+\dot Y\left(t_2^-(X,Y,t)\right)}{1-\dot Y\left(t_2^-(X,Y,t)\right)}\frac1{\left(X(t)-Y(t_2^-(X,Y,t))\right)^2} \\
 &+ \frac{1-\dot Y\left(t_2^+(X,Y,t)\right)}{1+\dot Y\left(t_2^+(X,Y,t)\right)}\frac1{\left(X(t)-Y(t_2^+(X,Y,t))\right)^2}\,, \\
 F_2(X,Y)(t):=&-\frac{1-\dot X\left(t_1^-(X,Y,t)\right)}{1+\dot X\left(t_1^-(X,Y,t)\right)}\frac1{\left(Y(t)-X(t_1^-(X,Y,t))\right)^2} \\
 &- \frac{1+\dot X\left(t_1^+(X,Y,t)\right)}{1-\dot X\left(t_1^+(X,Y,t)\right)}\frac1{\left(Y(t)-X(t_1^+(X,Y,t))\right)^2} 
 \end{split}
\end{equation}
with $t_{1/2}^\pm$ given by \eqref{t+-}, now depending on $X$ and $Y$ instead of  $a$ and $b$.
By definition, $H(0,\cdot)$ is the zero mapping and, if $(x,y)$ is a fixed point of $H(\lambda,\cdot)$, then
\begin{align*}
 X(t)=&a_0+\int_0^t\frac{P_1(\lambda,X,Y)(s)}{\sqrt{1+P_1(\lambda,X,Y)(s)^2}}\text ds \,,\\
 Y(t)=&b_0(T^+)+\int_0^{T^+}\frac{P_2(\lambda,X,Y)(s)}{\sqrt{1+P_2(\lambda,X,Y)(s)^2}}\text ds \,,
\end{align*}
which fulfill the equations
\begin{align*}
 \frac{\dot{X}(t)}{\sqrt{1-\dot{X}(t)^2}}=P_1(\lambda,X,Y)(t),\,\,\frac{\dot{Y}(t)}{\sqrt{1-\dot{Y}(t)^2}}=P_2(\lambda,X,Y)(t)\,,
\end{align*}
i.e.,
\begin{equation}\label{Fixpunkt}
\begin{split}
 \frac{\text d}{\text dt}\left(\frac{\dot{X}(t)}{\sqrt{1-\dot{X}(t)^2}}\right)=&\frac{\ddot X(t)}{(1-\dot X(t)^2)^{\frac32}} 
 = (1-\lambda)\frac{\ddot x_0(t)}{(1-\dot x_0(t)^2)^{\frac32}}+\frac{\lambda\kappa_a}2F_1(X,Y)(t)\,,\\
 \frac{\text d}{\text dt}\left(\frac{\dot{Y}(t)}{\sqrt{1-\dot{Y}(t)^2}}\right)=&\frac{\ddot Y(t)}{(1-\dot Y(t)^2)^{\frac32}} 
 = (1-\lambda)\frac{\ddot y_0(t)}{(1-\dot
     y_0(t)^2)^{\frac32}}+\frac{\lambda\kappa_b}2F_2(X,Y)(t).
 \end{split}
\end{equation}
Hence, for $\lambda=1$, $(X,Y)$ solve the equations of motion \eqref{WF}
respecting the initial data \eqref{Anfangswerte}.  In
Proposition~\ref{prop:Asymptotik+-} below we provide a priori
estimates ensuring that $H$ has no fixed points on the boundary of
\begin{equation}\label{Omega}
\begin{split}
 \Omega:=\bigg\{&(x,y)\in B\,\bigg|\, \|\dot X\|_\infty,\|\dot Y\|_\infty<V,\inf_{t\in[0,T]}(X(t)-Y(t))>D,
 \inf_{t\in[T,\infty[}(\dot X(t)-\dot Y(t))>v, \\
 &\sup_{t\ge0}(1+|t|)|\ddot X(t)|,\sup_{t\ge T^+}(1+|t|)|\ddot Y(t)|<A \bigg\}
 \end{split}
\end{equation}
for suitably chosen constants $v,V\in]0,1[$ and $A,D,T>0$; recall definition \eqref{XY} of $X,Y$ in terms of $x,y$.
The estimates for $\lambda\approx1$ will come from energy considerations, the
ones for $\lambda\approx0$ from the properties \eqref{x0y0} of $(x_0,y_0)$.

On $\overline\Omega$, it is rather straightforward to prove continuity and
compactness of $H$; see Lemma~\ref{lem:Stetigkeit} and
Lemma~\ref{lem:Kompaktheit} below.  Note that $\Omega$ is
open. In order to get an open set with decaying accelerations, which are crucial
for the compactness proof, we had to include the supremum norm of the
acceleration together with the blowup factor $1+|t|$ in the definition of
$\|\cdot\|$ in \eqref{Norm}.  To establish compactness of the range of $H$, we
basically apply the Arzela-Ascoli theorem on a finite time interval and take
into account that the trajectories have no large variations outside such an
interval due to the $\frac1{1+t^2}$ decay of the accelerations effected by the
Coulomb forces.  The choice of the norm forces us to apply the Arzela-Ascoli
theorem to $\dot x$ and $(1+|t|)\ddot x$. Hence, it is fortunate that the
$\frac1{1+|t|}$ a priori estimate given in \eqref{Schrankend} below suffices for the
continuity proof - had we needed $\frac1{1+t^2}$ there, we would have had to
take $1+t^2$ into the norm and to consider $(1+t^2)\ddot x$ in the compactness
proof, which is of order one on all of $\mathbb R_0^+$.

Now we perform the detailed calculations, starting with a technicality:
\begin{lemma}\label{lem:t+-}
 For any $C^1$-trajectories $(x,y)$ with $\|\dot x\|_\infty,\|\dot y\|_\infty\le C<1$ and $x(t)>y(t)$ for all $t\in\mathbb R$,
 the advanced and retarded times $t_i^\pm$ introduced in \eqref{t+-} are well-defined and
 \begin{equation*}
  \begin{split}
 &\frac{x(t)-y(t)}2\le x(t)-y(t_2^\pm(t))\le\frac{x(t)-y(t)}{1-\|\dot y\|_\infty}\,, \\
 &\frac{x(t)-y(t)}2\le x(t_1^\pm(t))-y(t)\le\frac{x(t)-y(t)}{1-\|\dot x\|_\infty} \,.
 \end{split}
 \end{equation*}
\end{lemma}
\begin{proof}
 Existence is clear;
 \begin{align*}
  x(t)-y(t_2^\pm)=&x(t)-y(t)+y(t)-y(t_2^\pm)=x(t)-y(t)+\dot y(\tau)(t-t_2^\pm) \\
  =&x(t)-y(t)\mp\dot y(\tau)[x(t)-y(t_2^\pm)]
 \end{align*}
with $\tau$ between $t$ and $t_2^\pm$, so
$$x(t)-y(t_2^\pm)=\frac{x(t)-y(t)}{1\pm\dot y(\tau)}\,.$$
\end{proof}

In particular, the lemma shows that $H$ is well-defined on
$$M:=\{(x,y)\in C^1(\mathbb R_0^+)\times C^1([T^-,\infty[)\mid \|\dot X\|_\infty,\|\dot Y\|_\infty<1\,,\,\,\forall_{t\in\mathbb R}X(t)> Y(t)\}\,.$$
After these preparations, we provide the required global a priori estimates on fixed points of $H$ in terms of the initial data. We partly adapt and generalize ideas from \cite{Bauer}.
\begin{proposition}
    \label{prop:Asymptotik+-}
 There are constants $\tilde v,\tilde V\in]0,1[$ and $\tilde A,\tilde D,\tilde T>0$ such that, for any fixed point $(x,y)$ of $H:[0,1]\times M\rightarrow C^1(\mathbb R_0^+)\times C^1([T^-,\infty[)$, 
  \begin{subequations}\label{Schranken}
  \begin{align}
  &\|\dot X\|_\infty,\|\dot Y\|_\infty<\tilde V\,\,, \label{Schrankena} \\
  &\inf_{t\ge0}(X-Y)(t)>\tilde D\,\,, \label{Schrankenb}\\
  &\inf_{t\ge \tilde T}(\dot X-\dot Y)(t)>\tilde v\,\,, \label{Schrankenc}\\
  &|\ddot X(t)|, |\ddot Y(t)|<\frac{\tilde A}{1+|t|}
  \,\,.\label{Schrankend}
  \end{align}
 \end{subequations}
\end{proposition}
We remark that this result implies in particular that any solution $(X,Y)$ of
\eqref{WF} with $\epsilon_\pm=\frac12$ satisfying \eqref{Anfangswerte}  obeys all
bounds \eqref{Schrankena}-\eqref{Schrankend}.  Physically, inequalities
\eqref{Schrankena} make sure that the velocities of the charges are bounded away
from one, and inequality
\eqref{Schrankenb} guarantees that the charges obey a minimal distance. More
precisely, inequality \eqref{Schrankenc} ensures that, at least starting from a
certain time, the charges move away from each other in the future, and
inequality \eqref{Schrankend} describes how the accelerations of the charges
decay and the trajectories relax to their corresponding incoming and outgoing
asymptotes. All possible fixed-points therefore describe scattering solutions.

The key in the following proof of the desired a priori bounds is a comparison
of the time derivative of the special relativistic kinetic energy with the one
of the potential energy carried in the delayed Coulomb potentials:
\begin{proof}
According to
\eqref{Fixpunkt} and \eqref{F+-},  the time derivative of the ``kinetic energy'' of the first component
of a fixed point of $H$ satisfies
 \begin{equation}\label{Energie}
 \begin{split}
  \frac{\text d}{\text dt}\left(\frac1{\sqrt{1-\dot X(t)^2}}\right)=&\frac{\dot X(t)\ddot X(t)}{(1-\dot X(t)^2)^{\frac32}} \\
  =& (1-\lambda)\frac{\dot X(t)\ddot x_0(t)}{(1-\dot x_0(t)^2)^{\frac32}}
   + \frac{\lambda\kappa_a}2\left[\frac{\dot X(t)+\dot X(t)\dot Y(t_2^-)}{1-\dot Y(t_2^-)}\frac1{(X(t)-Y(t_2^-))^2}\right. \\
   &+\left.\frac{\dot X(t)-\dot X(t)\dot Y(t_2^+)}{1+\dot Y(t_2^+)}\frac1{(X(t)-Y(t_2^+))^2}\right]\,.
  \end{split}
 \end{equation}
By definition \eqref{t+-} we have
\begin{equation}\label{tpunkt}
\dot t_2^\pm=1\pm\dot X(t)\mp\dot Y(t_2^\pm)\dot t_2^-=\frac{1\pm\dot X(t)}{1\pm\dot Y(t_2^\pm)}\,,
\end{equation}
so that  the derivative of the ``potential energy'' is given by
\begin{equation}\label{Vpunkt}
 \frac{\text d}{\text dt}\left(-\frac1{X(t)-Y(t_2^\pm)}\right)=\frac{\dot X(t)-\dot Y(t_2^\pm)\dot t_2^\pm}{[X(t)-Y(t_2^\pm)]^2}
 =\frac{\dot X(t)-\dot Y(t_2 ^\pm)}{1\pm\dot Y(t_2^\pm)}\frac1{[X(t)-Y(t_2^\pm)]^2}\,.
\end{equation}
If $\dot Y(s)$ happens to be nonpositive for all $s\in[T^-,T^+]$, then, since
the acceleration is always negative according to \eqref{Fixpunkt}, $\dot
Y(t_2^-)\le0$ for all $t\ge T^-$. Consequently, $\dot X(t)\mp\dot X(t)\dot
Y(t_2^\pm)\le\dot X(t)-\dot Y(t_2^\pm)$, and the last two terms on the
right-hand side of \eqref{Energie} can be estimated by a total differential
\eqref{Vpunkt}:
\begin{equation*}\label{kin}
 \frac{\text d}{\text dt}\left(\frac1{\sqrt{1-\dot X(t)^2}}\right)
 \le (1-\lambda)\frac{\dot X(t)\ddot x_0(t)}{(1-\dot x_0(t)^2)^{\frac32}}
   - \frac{\lambda\kappa_a}2\frac{\text d}{\text dt}\left[\frac1{X(t)-Y(t_2^-)}+\frac1{X(t)-Y(t_2^+)}\right]
\end{equation*}
Furthermore, $\dot X(t)$ in the first summand can be estimated by  one,
and integration from 0 to $t$ gives
\begin{align*}
 &\frac1{\sqrt{1-\dot X(t)^2}}-\frac1{\sqrt{1-\dot a_0^2}} \le (1-\lambda)\left[\frac{\dot x_0(t)}{\sqrt{1-\dot x_0(t)^2}} - \frac{\dot a_0}{\sqrt{1-\dot a_0^2}}\right] \\
 &+ \frac{\lambda\kappa_a}2\left[\frac1{a_0- b_0(T^-)}-\frac1{X(t)-Y(t_2^-)}
 +\frac1{a_0- b_0(T^+)}-\frac1{X(t)-Y(t_2^+)}\right]\,,
\end{align*}
 which implies
\begin{align*}
 \frac1{\sqrt{1-\dot X(t)^2}}\le& \frac2{\sqrt{1-(\dot a_0\vee\|\dot x_0\|_\infty)^2}}
 + \frac{\kappa_a}3\left[\frac1{a_0- b(T^-)}+\frac1{a_0- b(T^+)}\right] \\
 &-\frac{\lambda\kappa_a}2\left[\frac1{X(t)-Y(t_2^-)}+\frac1{X(t)-Y(t_2^+)}\right]\,.
\end{align*}
 Here, $a\vee b$ denotes the maximum of $a$ and $b$.
 Should we instead have $\dot Y(s)>0$ for some $s\in[T^-,T^+]$, the Lorentz
invariance of the equations of motion suggests to repeat the calculation for $\frac{1-\|\dot
b_0\|_\infty\dot X(t)}{\sqrt{1-\|\dot b_0\|_\infty^2}\sqrt{1-\dot X(t)^2}}$,
 which is the ``kinetic energy'' term
in an inertial frame with relative velocity $\|\dot b_0\|_\infty$:
\begin{align*}
 &\frac{\text d}{\text dt}\left(\frac{1-\|\dot b_0\|_\infty\dot X(t)}{\sqrt{1-\dot X(t)^2}}\right)
 = \frac{(\dot X(t)-\|\dot b_0\|_\infty)\ddot X(t)}{(1-\dot X(t)^2)^{\frac32}} \\
  =& (1-\lambda)\frac{(\dot X(t)-\|\dot b_0\|_\infty)\ddot x_0(t)}{(1-\dot x_0(t)^2)^{\frac32}}
   + \frac{\lambda\kappa_a}2\left[\frac{(\dot X(t)-\|\dot b_0\|_\infty)(1+\dot Y(t_2^-))}{1-\dot Y(t_2^-)}\frac1{(X(t)-Y(t_2^-))^2}\right. \\
   &+\left.\frac{(\dot X(t)-\|\dot b_0\|_\infty)(1-\dot Y(t_2^+))}{1+\dot Y(t_2^+)}\frac1{(X(t)-Y(t_2^+))^2}\right]\,.
\end{align*}
 We observe the inequality
\begin{equation}\label{Ypunktkleiner0}
 \begin{split}
 (\dot X(t)-\|\dot b_0\|_\infty)(1\mp\dot Y(t_2^\pm)) =& (1\mp\|\dot b_0\|_\infty)(\dot X(t)-\dot Y(t_2^\pm)) + (\dot Y(t_2^\pm)-\|\dot b_0\|_\infty)(1\mp\dot X(t)) \\
 \le& (1\mp\|\dot b_0\|_\infty)(\dot X(t)-\dot Y(t_2^\pm)),
 \end{split}
\end{equation}
 where $\dot Y(t_2^\pm)\le\|\dot b_0\|_\infty$ holds true. The latter can be seen
from the fact that it is satisfied for $t=0$ and that $\ddot Y<0$ implies this
also for all later times.  Using equations \eqref{tpunkt} and \eqref{Vpunkt}
again and upon integration from 0 to $t$, we now arrive at
\begin{align*}
 &\frac{1-\|\dot b_0\|_\infty\dot X(t)}{\sqrt{1-\dot X(t)^2}}-\frac{1-\|\dot b_0\|_\infty\dot a_0}{\sqrt{1-\dot a_0^2}} \le (1-\lambda)\left[\frac{\dot x_0(t)}{\sqrt{1-\dot x_0(t)^2}} - \frac{\dot a_0}{\sqrt{1-\dot a_0^2}}\right] \\
 &+ \frac{\lambda\kappa_a}2\left[\frac{1+\|\dot b_0\|_\infty}{a_0- b_0(T^-)}-\frac{1+\|\dot b_0\|_\infty}{X(t)-Y(t_2^-)}
 +\frac{1-\|\dot b_0\|_\infty}{a_0- b_0(T^+)}-\frac{1-\|\dot b_0\|_\infty}{X(t)-Y(t_2^+)}\right]\,.
\end{align*}
In consequence,  we get
\begin{equation}\label{Energie2}
\begin{split}
 \frac{1-\|\dot b_0\|_\infty}{\sqrt{1-\dot X(t)^2}}
 \le& \frac4{\sqrt{1-(\dot a_0)\vee\|\dot x_0\|_\infty)^2}} + \frac{\kappa_a}2\left[\frac2{a_0- b_0(T^-)} + \frac1{a_0- b(T^+)}\right] \\
 &- \frac{\lambda\kappa_a}2\left[\frac{1-\|\dot b_0\|_\infty}{X(t)-Y(t_2^-)} + \frac{1-\|\dot b_0\|_\infty}{X(t)-Y(t_2^+)}\right]\,.
 \end{split}
\end{equation}
Using Lemma \ref{lem:t+-}, we estimate $a_0-b_0(T^-)\ge\frac{a_0-b_0}{2}$ to find
\begin{equation*}
 \sup_{t\ge0}|\dot X(t)|\le\sqrt{1-\frac{(1-\|\dot b_0\|_\infty)^2}{\left(\frac{4}{\sqrt{1-(\dot a_0\vee\|\dot x_0\|_\infty)^2}}+\frac{3\kappa_a}{a_0-b_0(0)}\right)^2}} =: \tilde V_X \,.
\end{equation*}
The estimate $\sup_{t\ge T^+}|\dot Y(t)|\le\tilde V_Y$ is obtained in an analogous way. This concludes the proof of \eqref{Schrankena} for $\tilde V:=\tilde V_X\vee \tilde V_Y\vee\|\dot b_0\|_\infty$.

 In order to prove \eqref{Schrankenb} we go back to inequality \eqref{Energie2}.
First, we note that $X(t)-Y(t_2^\pm)\le\frac{X(t)-Y(t)}{1-\tilde V}$ according to
Lemma \ref{lem:t+-}. The resulting inequality ensures
\begin{equation}\label{Abstand}
 \inf_{t\ge0} \left(X(t)-Y(t)\right)\ge \frac{\lambda\kappa_a(1-\tilde
 V)^2}{\frac{4}{\sqrt{1-\dot a_0^2}}+\frac{3\kappa_a}{a_0-b_0(0)}}.
\end{equation}

This estimate has to be improved for small values of $\lambda$, but  before doing
that, it is more convenient to proceed with the estimates \eqref{Schrankenc} and \eqref{Schrankend} involving $\tilde T$ and $\tilde v$. Integrating the difference of the equations
\eqref{Fixpunkt} results in
\begin{align}
 \frac{\dot X(t)}{\sqrt{1-\dot X(t)^2}} - \frac{\dot Y(t)}{\sqrt{1-\dot Y(t)^2}}
 =& \lambda\left(\frac{\dot a_0}{\sqrt{1-\dot a_0^2}}-\frac{\dot b_0(0)}{\sqrt{1-\dot b_0(0)^2}}\right) \nonumber \\
 &+ (1-\lambda)\left(\frac{\dot x_0(t)}{\sqrt{1-\dot x_0(t)^2}}-\frac{\dot y_0(t)}{\sqrt{1-\dot y_0(t)^2}}\right) \nonumber \\
 &+\frac\lambda2\int_0^t(\kappa_aF_1(X,Y,s)-\kappa_bF_2(X,Y,s))\text ds \label{gu-v}\\
 \ge&-\frac4{\sqrt{1-\tilde V^2}}+\frac\lambda2\int_0^t(\kappa_aF_1(X,Y,s)-\kappa_bF_2(X,Y,s))\text ds \,. \label{uu-v}
\end{align}
 The function $\frac u{\sqrt{1-u^2}}$ is strictly monotone in $u$, so $a>b$ is
equivalent to $\frac{a}{\sqrt{1-a^2}}>\frac{b}{\sqrt{1-b^2}}$. Furthermore,
because of $\ddot
x_0\ge0$ and $\ddot y_0\le0$, $F_1$ is positive and $F_2$ negative according to
\eqref{F+-}. Therefore, equation \eqref{gu-v} shows that, if $\dot a_0\ge\dot b_0(0)$, then
$\dot X(t)\ge\dot Y(t)$ holds for all larger $t$. On the contrary, if $\dot a_0<\dot b_0(0)$ and $\dot X(t)<\dot Y(t)$, then
$$X(s)-Y(s)=a_0-b_0(0)+\int_0^s(\dot X(r)-\dot Y(r))\text dr<a_0-b_0(0)$$ for
all $s\in]0,t]$.  Thus, we find
$$|F_i(X,Y,s)|>\frac{(1-\tilde V)^2}{a_0-b_0(0)}\,.$$
That means, for
$$t\ge S_1(\lambda):=\frac{8(a_0-b_0(0))}{\lambda(\kappa_a+\kappa_b)\sqrt{1-\tilde V^2}(1-\tilde V)^2}\,,$$
\eqref{uu-v} implies $\dot X(t)-\dot Y(t)\ge0$. Recall further that, from time
$T^++1$ on, the accelerations of $x_0$ and $y_0$ are zero.  Hence, again by \eqref{Fixpunkt}, for $t\ge
(S_1(\lambda)\vee T^+)+2$ we have
\begin{equation*}\label{u2u-v}
 \begin{split}
  \frac{\dot X(t)}{\sqrt{1-\dot X(t)^2}} - \frac{\dot Y(t)}{\sqrt{1-\dot Y(t)^2}}
  \ge\frac\lambda2\int_{\left(S_1(\lambda)\vee T^+\right)+1}^{(S_1(\lambda)\vee T^+)+2}(\kappa_aF_1(X,Y,s)-\kappa_bF_2(X,Y,s))\text ds\,.
 \end{split}
\end{equation*}
Since $\frac{\text d}{\text du}\frac u{\sqrt{1-u^2}}=\frac1{(1-u^2)^{\frac32}}\le\frac1{(1-\tilde V^2)^{\frac32}}$,
it holds
$$\dot X(t)-\dot Y(t)\ge(1-\tilde V^2)^{\frac32}\left(\frac{\dot
X(t)}{\sqrt{1-\dot X(t)^2}} - \frac{\dot Y(t)}{\sqrt{1-\dot Y(t)^2}}\right)\,;$$
moreover, we have $X(s)-Y(s)\le a_0-b_0(0)+2\tilde Vs$, so that we arrive at
$$|F_i(X,Y,s)|\ge\frac{(1-\tilde V)^2}{a_0-b_0(0)+2\tilde Vs}\,.$$
 We therefore infer
\begin{equation*}
 \inf_{t\ge (S_1(\lambda)\vee T^+)+2}\left(\dot X(t)-\dot Y(t)\right)\ge v_1(\lambda):=\frac{\lambda(1-\tilde V^2)^{\frac32}}2\int_{(S_1(\lambda)\vee T^+)+1}^{(S_1(\lambda)\vee T^+)+2}\frac{(\kappa_a+\kappa_b)(1-\tilde V)^2}{2(a_0-b_0+2\tilde Vs)}\text ds\,.
\end{equation*}

 To get our hands on the missing estimates for small values of $\lambda$ we rewrite
and estimate the integrated version of equation \eqref{Fixpunkt} for the
velocity according to
\begin{align*}
 \frac{\dot X(t)}{\sqrt{1-\dot X(t)^2}}-\frac{\dot x_0(t)}{\sqrt{1-\dot x_0(t)^2}} =& \lambda\left(\frac{\dot a_0}{\sqrt{1-\dot a_0^2}} - \frac{\dot x_0(t)}{\sqrt{1-\dot x_0(t)^2}}\right)
 +\frac{\lambda\kappa_a}2\int_0^tF_1(X,Y,s)\text ds \\
 \ge&-\frac{2\lambda}{\sqrt{1-\tilde V^2}},
\end{align*}
and conclude
\begin{equation*}
 \dot X(t)-\dot x_0(t)\ge -\frac{2\lambda}{\sqrt{1-\tilde V^2}} \,.
\end{equation*}
 The latter clearly holds true if $\dot X(t)-\dot x_0(t)$ is non-negative, and in the opposite
case, it follows from the fact that $\frac{\text d}{\text du}\frac{u}{\sqrt{1-u^2}}=\frac1{(1-u^2)^{\frac32}}\ge1$.
Consequently, we have $$X(t)\ge x_0(t)-\frac{2\lambda t}{\sqrt{1-\tilde V^2}}\,;$$
 similarly we find the corresponding estimates for the other charge:
$$\dot Y(t)-\dot y_0(t)\le \frac{2\lambda}{\sqrt{1-\tilde V^2}} \qquad
\text{and}
\qquad Y(t)\le y_0(t)+\frac{2\lambda t}{\sqrt{1-\tilde V^2}}\,.$$
 Hence, the last four estimates imply
\begin{equation}\label{Abstand2}
 X(t)-Y(t)\ge \inf_{t\ge0}\left(x_0(t)-y_0(t)\right)-\frac{4\lambda|t|}{\sqrt{1-\tilde V^2}}
\end{equation}
and
\begin{align*}
 \inf_{t\ge T^++1}\left(\dot X(t)-\dot Y(t)\right)\ge& \inf_{t\ge T^++1}\left(\dot x_0(t)-\dot y_0(t)\right)-\frac{4\lambda}{\sqrt{1-\tilde V^2}} \\
 \ge& (\dot x_0-\dot y_0)(T^++1) -\frac{4\lambda}{\sqrt{1-\tilde
 V^2}}=:v_2(\lambda),
\end{align*}
 where in the last inequality we used the properties of $x_0$ and $y_0$
given in \eqref{x0y0}.

Combining the estimate $$\inf_{t\ge T^++1}\left(\dot X(t)-\dot Y(t)\right)\ge
v_2(\lambda^\prime)>0$$ for all fixed points corresponding to values
$$\lambda\le\lambda^\prime:=\frac{\sqrt{1-\tilde V^2}}{8}(\dot x_0-\dot y_0)(T^++1)\,,$$
with the estimate $$\inf_{t\ge (S_1(\lambda^\prime)\vee T^+)+2}\left(\dot
X(t)-\dot Y(t)\right)\ge v_1(\lambda^\prime)>0$$  that holds for all fixed points
corresponding to some $\lambda\ge\lambda^\prime$, we obtain $\tilde T$ and
$\tilde v$. Furthermore, equation \eqref{Abstand2} gives a positive lower bound
for the distance $X(t)-Y(t)$ up to time $t=\tilde T$ if $(X,Y)$ is a fixed point
corresponding to a $\lambda$ fulfilling
$$\lambda\le\frac{\sqrt{1-\tilde V^2}\inf_{t\ge0}(x_0(t)-y_0(t))}{16\tilde T}\,.$$
This estimate extends to all times and, together with the above estimate
\eqref{Abstand} corresponding to larger values of $\lambda$, yields the constant
$\tilde D$, which concludes the proof of inequalities \eqref{Schrankenb} and
\eqref{Schrankenc}.

The existence of $\tilde D$ and $\tilde v$ imply $X(t)-Y(t)\ge D_1+v_3|t|$ for
suitable $D_1\le\tilde D$ and $v_3\le\tilde v$, so that an estimate for $\ddot
X$ and $\ddot Y$,  which proves the remaining bound in \eqref{Schrankend}, can now be directly read off
equation \eqref{Fixpunkt}:
$$|\ddot X(t)|,|\ddot Y(t)|\le\left[\frac{\|\ddot x_0\|_\infty\vee\|\ddot y_0\|_\infty}{(1-\tilde V^2)^{\frac32}} + \frac{2(\kappa_a+\kappa_b)}{(1-\tilde V)(D_1+v_3|t|)^2}\right]\vee\sup_{T^-\le t\le T^+}|\ddot b_0|
\le\frac{\tilde A}{1+|t|}\,\,.$$

\end{proof}

For the proof of Theorem \ref{thm:+-}, which is given at the end of this
section, we need the following technical lemmata. In their proofs, the symbols
$C,C_1,C_2,\dots$ will be used for positive constants whose values may vary from
line to line.

\begin{lemma}\label{lem:Stetigkeit}
$H$ is Lipschitz continuous on $[0,1]\times\overline\Omega$.
\end{lemma}
\begin{proof}
We will repeatedly use the following decay property of the  Coulomb-like force
term in $F$ given in \eqref{F+-}: Since,
according to definition~\eqref{Omega} of $\Omega$,
$|X(t)-Y(t)|\ge D$ holds for all $t$ and $|X(t)-Y(t)|\ge D+v|t-T|$ does for
$|t|\ge T$, we get the estimate
$|X(t)-Y(t)|\ge C(1+|t|)$ for all $t$. Using  Lemma \ref{lem:t+-},
\begin{equation}\label{wachsx}
|X(t)-Y(t_2^-(t))|,|X(t_1^-(t))-Y(t)|\ge C(1+|t|) 
\end{equation}
 is satisfied, and, together with the velocity bound
\begin{equation}\label{V}
|\dot X(t)|,|\dot Y(t)|\le V
\end{equation}
from the definition of $\Omega$, the definition \eqref{F+-} of $F$ leads to
\begin{equation}\label{abfallF}
 \begin{split}
  |F_i(X,Y)(t)|\le\frac{C}{1+t^2}\,,
 \end{split}
\end{equation}
which ensures that $H$ maps into $B$.

 To show Lipschitz continuity, we choose $(\lambda,x,y),(\tilde\lambda,\tilde x,\tilde
y)\in[0,1]\times\overline\Omega$ and $t\in\mathbb R$ and, recalling definition \eqref{H} of $H$, consider
\begin{equation}\label{hpunkt1}
 \begin{split}
 &|\dot H_1(\lambda,x,y)(t)-\dot H_1(\tilde\lambda,\tilde x,\tilde y)(t)| \\
 =&\left|\frac{P_1(\lambda,X,Y)(t)}{\sqrt{1+P_1(\lambda,X,Y)(t)^2}}-\frac{P_1(\tilde\lambda,\tilde X, \tilde Y)(t)}{\sqrt{1+P_1(\tilde\lambda,\tilde X, \tilde Y)(t)^2}}\right| \\
 \le&\left|P_1(\lambda,X,Y)(t)-P_1(\tilde\lambda,\tilde X, \tilde Y)(t)\right|\,.
 \end{split}
 \end{equation}
 This inequality holds because $\left|\frac{\text d}{\text du}\frac{u}{\sqrt{1+u^2}}\right|=\left|\frac1{(1+u^2)^{\frac32}}\right|\le1$.
 Substituting definition \eqref{P} of $P_1$, splitting the summands via triangle
 inequality, and using the velocity bound \eqref{V} as well as estimate \eqref{abfallF} on $F$,
  we find
 \begin{equation}\label{hpunkt2}
 \begin{split}
 &\left|P_1(\lambda,X,Y)(t)-P_1(\tilde\lambda,\tilde X, \tilde Y)(t)\right| \\
 \le& \frac{\dot x_0(t)}{\sqrt{1-\dot x_0(t)^2}}|\lambda-\tilde\lambda| + \frac{\dot a_0}{\sqrt{1-\dot a_0^2}}|\lambda-\tilde\lambda|
 + \int_0^t\frac{\kappa_a}2|F_1(X,Y)(s)|\text ds|\lambda-\tilde\lambda| \\
 &+ |\lambda|\int_0^t\frac{\kappa_a}2|F_1(X,Y)(s)-F_1(\tilde X, \tilde Y)(s)|\text ds \\
 \le& \frac V{\sqrt{1-V^2}}|\lambda-\tilde\lambda|+\frac V{\sqrt{1-V^2}}|\lambda-\tilde\lambda|+\int_0^\infty\frac{\kappa_aC}{1+t^2}\text dt|\lambda-\tilde\lambda| \\
 &+\int_0^\infty\frac{\kappa_a}2|F_1(X,Y)(t)-F_1(\tilde X, \tilde Y)(t)|\text dt \\
 \le& C|\lambda-\tilde\lambda| + \int_{-\infty}^\infty\kappa_a|F_1(X,Y)(t)-F_1(\tilde X, \tilde Y)(t)|\text dt \,.
 \end{split}
\end{equation}
 To keep the notation short, we write 
\begin{align*}
&\tilde x+x_0=\tilde X\,,\,\,\tilde y+y_0=\tilde Y\,,\\
&t_{1/2}^\pm(X,Y,t)=t_{1,2}^\pm\,,\,\,t_{1/2}^\pm(\tilde X, \tilde Y,t)=\tilde t_{1,2}^\pm\,\\
&P_1(\lambda, X, Y)(t)=P(t)\,,\,\,P_1(\tilde\lambda,\tilde X, \tilde Y)(t)=\tilde P(t)\,, \\
&F_1(X, Y)(t)=F(t)\,,\,\,F_1(\tilde X, \tilde Y)(t)=\tilde F(t)\,,
\end{align*}
from now on. We continue with
\begin{equation*}
 \begin{split}
 |F(t)-\tilde F(t)|
 \le& \left|\frac{1-\dot Y\left(t_2^+\right)}{1+\dot Y\left(t_2^+\right)}
 -\frac{1-\dot{\tilde Y}\left(\tilde t_2^+\right)}{1+\dot{\tilde Y}\left(\tilde t_2^+\right)}\right|\frac1{\left(X(t)-Y(t_2^+)\right)^2} \\
 &+ \left|\frac{1+\dot Y\left(t_2^-\right)}{1-\dot Y\left(t_2^-\right)}
 -\frac{1+\dot{\tilde Y}\left(\tilde t_2^-\right)}{1-\dot{\tilde Y}\left(\tilde t_2^-\right)}\right|\frac1{\left(X(t)-Y(t_2^-)\right)^2} \\
 &+ \frac{1-\dot{\tilde Y}\left(\tilde t_2^+\right)}{1+\dot{\tilde Y}\left(\tilde t_2^+\right)}
 \left|\frac1{\left(X(t)-Y(t_2^+))\right)^2}-\frac1{\left(\tilde X(t)-\tilde Y(\tilde t_2^+)\right)^2}\right| \\
 &+ \frac{1+\dot{\tilde Y}\left(\tilde t_2^-\right)}{1-\dot{\tilde Y}\left(\tilde t_2^-\right)}
 \left|\frac1{\left(X(t)-Y(t_2^-))\right)^2}-\frac1{\left(\tilde X(t)-\tilde Y(\tilde t_2^-)\right)^2}\right| \,.
 \end{split}
\end{equation*}
 Exploiting $\frac{\text d}{\text du}\left(\frac{1\mp u}{1\pm
u}\right)=\mp\frac2{(1\pm u)^2}$, $\frac{\text d}{\text
du}\left(\frac1{u^2}\right)=-\frac2{u^3}$, \eqref{wachsx}, and $\|\dot
y\|_\infty,\|\dot{\tilde y}\|_\infty\le V$ leads to
\begin{equation}\label{F}
 \begin{split}
 &|F(t)-\tilde F(t)| \\
 \le& \frac2{(1-V)^2}\left[|\dot Y\left(t_2^+\right)-\dot{\tilde Y}\left(\tilde t_2^+\right)|+|\dot Y\left(t_2^-\right)-\dot{\tilde Y}\left(\tilde t_2^-\right)|\right]\frac C{1+t^2}  \\
 &+\frac{1+V}{1-V}\frac{2C}{1+|t|^3}\left[|x(t)-\tilde x(t)|+|Y(t_2^+)-\tilde Y(\tilde t_2^+)|+|Y(t_2^-)-\tilde Y(\tilde t_2^-)|\right] \,.
 \end{split}
\end{equation}
According to definition \eqref{t+-}  we have $|t_2^\pm|\le
|t|+|a_0|+|b_0(0)|+V|t|+V|t_2^\pm|$, so that $$|t_2^\pm|\le C(1+|t|)$$ holds.
 Using this bound and, again, the definition of $t_2^\pm$ and $\tilde t_2^\pm$, we arrive at
\begin{equation}\label{ytilde}
\begin{split}
 |Y(t_2^\pm)-\tilde Y(\tilde t_2^\pm)| 
 \le&|y(t_2^\pm)-\tilde y(t_2^\pm)|+|\tilde Y(t_2^\pm)-\tilde Y(\tilde t_2^\pm)| \\
 \le&\|\dot y-\dot{\tilde y}\|_\infty|t_2^\pm|+|\tilde Y(t_2^\pm)-\tilde Y(\tilde t_2^\pm)| \\
 \le&\|\dot y-\dot{\tilde y}\|_\infty|t_2^\pm| + V|t_2^\pm-\tilde t_2^\pm| \\
 \le& C(1+|t|)\|\dot y-\dot{\tilde y}\|_\infty
 +V\left(|x(t)-\tilde x(t)|+|\color{blue} Y(t_2^\pm)-\tilde Y(\tilde t_2^\pm)|\right)\,.
 \end{split}
\end{equation}
Rearranging terms, we find
\begin{align}
    |\color{blue} Y(t_2^\pm)-\tilde Y(\tilde t_2^\pm)|
 \le& \frac C{1-V}(1+|t|)\|\dot y-\dot{\tilde y}\|_\infty
 +\frac V{1-V}|x(t)-\tilde x(t)| \nonumber \\
 \le& C(1+|t|)\left(\|\dot y-\dot{\tilde y}\|_\infty+\|\dot x-\dot{\tilde x}\|_\infty\right)\,.
\label{y-ytilde}
\end{align}
This bound and the definition of $t_2^\pm$ and $\tilde t_2^\pm$ imply
\begin{equation*}
 |t_2^\pm-\tilde t_2^\pm|
 \le C(1+|t|)\left(\|\dot y-\dot{\tilde y}\|_\infty+\|\dot x-\dot{\tilde
 x}\|_\infty\right).
\end{equation*}

Therefore, using the decay of the accelerations according to the definition of
$\Omega$, it holds
\begin{equation*}
\begin{split}
  |\dot Y\left(t_2^\pm\right)-\dot{\tilde Y}\left(\tilde t_2^\pm\right)|
  \le& |\dot Y\left(t_2^\pm\right)-\dot{\tilde Y}\left(t_2^\pm\right)| + |\dot {\tilde Y}\left(t_2^\pm\right)-\dot{\tilde Y}\left(\tilde t_2^\pm\right)| \\
  \le& \|\dot y-\dot{\tilde y}\|_\infty + \|\ddot{\tilde Y}\|_\infty|t_2^\pm-\tilde t_2^\pm| \\
  \le& C\left(\|\dot y-\dot{\tilde y}\|_\infty+\|\dot x-\dot{\tilde
  x}\|_\infty\right).
 \end{split}
\end{equation*}
 Coming back to \eqref{F}, we exploit the last bound 
together with \eqref{y-ytilde} and $|x(t)-\tilde x(t)|\le
|t|\|\dot x-\dot{\tilde x}\|_\infty$ to find  
\begin{equation}\label{F2}
 |F(t)-\tilde F(t)|\le\frac C{1+t^2}\|(x,y)-(\tilde x,\tilde y)\|
\end{equation}
and, by \eqref{hpunkt1} and \eqref{hpunkt2},
\begin{equation}\label{hpunkt3}
 \begin{split}
 |\dot H_1(\lambda,x,y)(t)-\dot H_1(\tilde\lambda,\tilde x,\tilde y)(t)|
 \le&\left|P(t)-\tilde P(t)\right|
 \le C_1|\lambda-\tilde\lambda|+C_2\|(x,y)-(\tilde x,\tilde y)\|\,.
 \end{split}
\end{equation}
To complete the proof, we have to take care of the second derivatives of $H$:\\
\begin{equation}\label{Hppunkt}
 \begin{split}
&|\ddot H_1(\lambda,x,y)(t)-\ddot H_1(\tilde\lambda,\tilde x,\tilde y)(t)|
= \left|\frac{\text d}{\text dt}\left(\frac{P(t)}{\sqrt{1-P(t)^2}}-\frac{\tilde P(t)}{\sqrt{1-\tilde P(t)^2}}\right)\right| \\
=& \left|\frac{\dot P(t)}{(1+P(t)^2)^{\frac 32}} - \frac{\dot{\tilde P}(t)}{(1+\tilde P(t)^2)^{\frac 32}}\right| \\
\le& |\dot P(t)|\cdot\left|\frac1{(1+P(t)^2)^{\frac 32}}-\frac1{(1+\tilde
    P(t)^2)^{\frac 32}}\right| + \frac1{(1+\tilde P(t)^2)^{\frac
    32}}\left|\dot P(t)-\dot{\tilde P}(t)\right|\,.
    \end{split}
\end{equation}
Since $|\frac{\text d}{\text
du}\frac1{(1+u^2)^{\frac32}}|=|-\frac{3u}{(1+u^2)^{\frac52}}|\le1$ for $|u\le1|$
and, by the definition of $P$ and the decay of $F$ given in \eqref{abfallF}, we
get
\begin{equation}\label{Pschranke}
 |P(t)|\le C_1+C_2\int_0^t|F(s)|\text ds\le C_1+C_2\int_0^\infty\frac{\text
 ds}{1+s^2}\le C.
\end{equation}
 Furthermore, since $\ddot x_0(t)=0$ for sufficiently large $|t|$ as can be
inferred from \eqref{x0y0},  
\begin{equation}\label{PPunktschranke}
 |\dot P(t)|\le(1-\lambda)\left|\frac{\ddot x_0(t)}{(1-\dot x_0(t))^{\frac32}}\right|+\lambda|F(t)|\le\frac C{1+t^2}
\end{equation}
holds.
 Thus, inequality \eqref{Hppunkt} results in
\begin{align*}
 &|\ddot H_1(\lambda,x,y)(t)-\ddot H_1(\tilde\lambda,\tilde x,\tilde y)(t)|
 \le\frac{C_1}{1+t^2}|P(t)-\tilde P(t)|+C_2|\dot P(t)-\dot{\tilde P}(t)|\,.
\end{align*}
Here, estimate \eqref{hpunkt3} can be used for the first modulus and
\begin{align*}
 |\dot P(t)-\dot{\tilde P}(t)|=&\kappa_a\left|\lambda F(t)\text ds-\tilde\lambda\tilde F(t)\right| \\
 \le&|\lambda-\tilde\lambda|C_1|F(t)|
 +\tilde\lambda|F(t)-\tilde F(t)| \\
 \le&\frac{C_1}{1+t^2}|\lambda-\tilde\lambda| +\frac{C_2}{1+t^2}\|(x,y)-(\tilde x,\tilde y)\|
 \end{align*}
according to bound \eqref{F2}. Therefore,  it holds that
\begin{align*}
 (1+|t|)|\ddot H_1(\lambda,x,y)(t)-\ddot H_1(\tilde\lambda,\tilde x,\tilde y)(t)|
 \le\frac{C_1}{1+|t|}|\lambda-\tilde\lambda|+\frac{C_2}{1+|t|}\|(x,y)-\tilde x,\tilde y\|\,.
\end{align*}
Together with \eqref{hpunkt3} and the corresponding estimates for $H_2$,  we
finally find
\begin{equation*}
 \|H(\lambda,x,y)(t)-H(\tilde\lambda,\tilde x,\tilde y)(t)\|\le C_1|\lambda-\tilde\lambda|+C_2\|(x,y)-(\tilde x,\tilde y)\|\,,
\end{equation*}
i.e., that $H$ is Lipschitz continuous, which was to show.
\end{proof}

\begin{lemma}\label{lem:Kompaktheitsbed}
 If a sequence $f_1,f_2,\dots$ of bounded continuous functions on $\mathbb R$ is uniformly bounded and equicontinuous and
 \begin{equation}\label{Kompaktheitsbed}
 \lim_{S\to\infty}\sup_{t>S,n\in\mathbb N}\left(|f_n(t)-f_n(S)|\vee|f_n(-t)-f_n(-S)|\right)=0\,,
 \end{equation}
 then it has a uniformly convergent subsequence.
\end{lemma}
\begin{proof}
 Let $(S_k)$ be a sequence such that $$\sup_{t>S_k,n\in\mathbb
 N}\left(|f_n(t)-f_n(S_k)|\vee|f_n(-t)-f_n(-S_k)|\right)\le\frac1k$$ for all
 $k\in\mathbb N$. The Arzela-Ascoli Theorem guarantees the existence of a
 subsequence $(f_{n^1_l})$ such that, for all $l\in\mathbb N$, $\sup_{|t|\le
 S_1}|f_{n^1_l}(t)-f_{n^1_m}(t)|<1$ for all $l,m\in\mathbb N$. For $|t|>S_1$, we
 find
 $$|f_{n^1_l}(t)-f_{n^1_m}(t)|\le|f_{n^1_l}(t)-f_{n^1_l}(- S_1)|+|f_{n^1_l}(-
 S_1)-f_{n^1_m}(- S_1)|+|f_{n^1_m}(S_1)-f_{n^1_m}(t)|\le3\,$$ so that
 $\|f_{n^1_l}-f_{n^1_m}\|_\infty<3$ holds.
 This construction can be iterated for every $k\in\mathbb N$ to yield a
 subsequence
 $\left(f_{n^k_l}\right)_{l\in\mathbb N}$ extracted from
 $\left(f_{n^{k-1}_l}\right)_l$ such that
 $\sup_{l,m\in\mathbb N}\|f_{n^k_l}-f_{n^k_m}\|_\infty<\frac3k$. The
 diagonal sequence $\left(f_{n^k_k}\right)_{k\in\mathbb N}$ is then uniformly
 Cauchy, and thus, uniformly convergent.
\end{proof}
\begin{lemma}\label{lem:Kompaktheit}
$H([0,1]\times\overline\Omega)$ is precompact in $B$.
\end{lemma}
\begin{proof}
Let $(\lambda_n,x_n,y_n)$ be a sequence in $[0,1]\times\overline\Omega$.  It is
convenient to introduce the following abbreviations: $H^{(n)}:=H(\lambda_n,x_n,y_n)$,
$P^{(n)}:=P(\lambda_n,x_n,y_n)$, and $F^{(n)}:=F(\lambda_n,x_n,y_n)$.
According to the definition of the norm on $B$
in \eqref{Norm}, we have to show
the existence of uniformly convergent subsequences of $(\dot H^{(n)})$ and
$((1+|\cdot|)\ddot H^{(n)})$. We will use the preceding Lemma \ref{lem:Kompaktheitsbed} and carry out the proof for $(\dot H_1^{(n)})$ and $((1+|\cdot|)\ddot H_1^{(n)})$ only as $(\dot H_2^{(n)})$ and
$((1+|\cdot|)\ddot H_2^{(n)})$ can be treated analogously.
Using the bound on $|P_1^{(n)}(t)|$ from \eqref{Pschranke}, we infer from definition \eqref{H} of $H$ that
\begin{align*}
 |\dot H_1^{(n)}(t)|\le\|\dot x_0\|_\infty + \left\|\frac{P_1^{(n)}}{\sqrt{1+\left(P_1^{(n)}\right)^2}}\right\|_\infty\le 2
\end{align*}
and
\begin{align}\label{H..}
 \ddot H_1^{(n)}(t)=-\ddot x_0(t)+\frac{\dot P_1^{(n)}(t)}{(1+P_1^{(n)}(t)^2)^{\frac32}}\,.
\end{align}
Together with the bound on $|\dot P_1^{(n)}(t)|$ from \eqref{PPunktschranke} and
the fact that $\ddot x_0(t)=0$ for sufficiently large $t$ from definition
\eqref{x0y0}  we infer
\begin{equation}\label{H..schranke}
\begin{split}
 (1+|t|)|\ddot H_1^{(n)}(t)| \le& (1+|t|)\left(|\ddot x_0(t)| + |\dot P_1^{(n)}(t)|\right) \\
 \le&\frac C{1+|t|}\,.
 \end{split}
\end{equation}
This means that both $(\dot H_1^{(n)})$ and $((1+|\cdot|)\ddot H_1^{(n)})$ are uniformly bounded; using in addition that
$\left|\frac{\text d}{\text du}\frac{u}{\sqrt{1+u^2}}\right|=\left|\frac1{(1+u^2)^{\frac32}}\right|\le1$ and
$\left|\frac{\text d}{\text du}\frac{u}{\sqrt{1-u^2}}\right|=\left|\frac1{(1-u^2)^{\frac32}}\right|\le\frac1{(1-V^2)^{\frac32}}$ for $|u|\le V$, we deduce for $t\ge s$
\begin{equation}\label{Hpunkt}
\begin{split}
 |\dot H_1^{(n)}(t)-\dot H_1^{(n)}(s)|\le& |\dot x_0(t)-\dot x_0(s)| + \left|P_1^{(n)}(t)-P_1^{(n)}(s)\right| \\
 \le& C|\dot x_0(t)-\dot x_0(s)| + \frac{\kappa_a}2\int_s^t|F_1^{(n)}(r)|\text dr \\
 \le& C_1\|\ddot x_0\|_\infty|t-s| + C_2\int_s^t\frac1{1+r^2}\text dr.
 \end{split}
\end{equation}
 Recalling \eqref{H..}, using the bounds on $|P_1^{(n)}(t)|$ and $|\dot
P_1^{(n)}(t)|$ from \eqref{Pschranke} and \eqref{PPunktschranke}, and employing the mean
value theorem (similarly as for inequality \eqref{Hppunkt} in the proof of Lemma
\ref{lem:Stetigkeit}), we obtain
\begin{equation}\label{H..diff}
 \begin{split}
 &|\ddot H_1^{(n)}(t)-\ddot H_1^{(n)}(s)| \\
 \le& |\dot P_1(t)|\cdot\left|\frac1{(1+P_1(t)^2)^{\frac 32}}-\frac1{(1+ P_1(s)^2)^{\frac 32}}\right|
 + \frac1{(1+ P_1(s)^2)^{\frac 32}}\left|\dot P_1(t)-\dot{ P_1}(s)\right| \\
 \le& \frac{C_1}{1+s^2}|P_1^{(n)}(t)-P_1^{(n)}(s)|+C_2|\dot P_1^{(n)}(t)-\dot P_1^{(n)}(s)|\\
 \le& \frac{C_1}{1+s^2}\int_s^t\frac1{1+r^2}\text dr
 + C_2\left[\left|\frac{\ddot x_0(t)}{(1-\dot x_0(t)^2)^{\frac32}}-\frac{\ddot x_0(s)}{(1-\dot x_0(s)^2)^{\frac32}}\right|+\frac{\kappa_a}2|F_1^{(n)}(t)-F_1^{(n)}(s)|\right]\,.
\end{split}
\end{equation}
 The summand in front of the bracket and the first one within the bracket 
can both be estimated by $\frac{C|t-s|}{1+s^4}$ because  $\ddot x_0$ was chosen uniformly continuous
and zero for large $t$. Moreover, using definition \eqref{F+-} of $F_1^{(n)}$,
splitting the summands via triangle inequality, applying the mean value theorem,
and taking into account the bounds \eqref{wachsx} and \eqref{V} on
$|X(t)-Y(t_2^\pm)|$ and $|\dot Y(t)|$, respectively,
we get
\begin{align*}
 &|F_1^{(n)}(t)-F_1^{(n)}(s)| \\
 \le& \frac{C_1}{1+s^2}\left|\dot Y^{(n)}(t_2^+(t))-\dot Y^{(n)}(t_2^+(s))\right|
 +\frac{C_2}{1+s^2}\left|\dot Y^{(n)}(t_2^-(t))-\dot Y^{(n)}(t_2^-(s))\right| \\
 &+ \frac{C_3}{1+|s|^3}\left|Y^{(n)}(t_2^+(t))- Y^{(n)}(t_2^+(s))\right|
 +\frac{C_4}{1+|s|^3}\left|Y^{(n)}(t_2^-(t))- Y^{(n)}(t_2^-(s))\right| \\
 &+\frac{C_5}{1+|s|^3}\left|X^{(n)}(t)- X^{(n)}(s)\right| \\
 \le&\frac{C}{1+|s|^3}\left(|t_2^+(t)-t_2^+(s)|+|t_2^-(t)-t_2^-(s)|+|t-s|\right)\,.
\end{align*}
Going back to definition \eqref{t+-} of $t_2^\pm$,  we observe
\begin{align*}
 |t_2^\pm(t)-t_2^\pm(s)|\le& |t-s|+|X^{(n)}(t)-X^{(n)}(s)|+\left|Y^{(n)}(t_2^\pm(t))- Y^{(n)}(t_2^\pm(s))\right| \\
 \le&|t-s|+C|t_2^\pm(t)-t_2^\pm(s)|\le C|t-s|\,,
\end{align*}
 so that
$$|F_1^{(n)}(t)-F_1^{(n)}(s)|\le\frac{C|t-s|}{1+|s|^3}$$
 is satisfied. Substituting this into \eqref{H..diff} gives
\begin{equation*}
 |\ddot H_1^{(n)}(t)-\ddot H_1^{(n)}(s)|\le\frac{C|t-s|}{1+|s|^3} \,.
\end{equation*}
Consequently, recalling estimate \eqref{H..schranke} for $|\ddot H_1^{(n)}(t)|$,
 we find
\begin{equation*}
 \begin{split}
 |(1+|t|)\ddot H_1^{(n)}(t)-(1+|s|)\ddot H_1^{(n)}(s)|
 \le& |t-s||\ddot H_1^{(n)}(t)|+(1+|s|)|\ddot H_1^{(n)}(t)-\ddot H_1^{(n)}(s)| \\
 \le& \frac{C|t-s|}{1+s^2}\,.
\end{split}
\end{equation*}
This, together with \eqref{Hpunkt}, implies equicontinuity and condition
\eqref{Kompaktheitsbed} for both $(\dot H_1^{(n)})$ and $((1+|\cdot|)\ddot
H_1^{(n)})$ as required by  Lemma \ref{lem:Kompaktheitsbed}.
\end{proof}
\begin{lemma}\label{lem:Regularitaet}
 $a$ is $(r\wedge s)+1$ times differentiable at $t$ if $b$ is $r$ times differentiable at $t_2^-(t)$ and $s$ times differentiable at $t_2^+(t)$ for $r,s\in\mathbb N$. Correspondingly, $b$ is $(r\wedge s)+1$ times differentiable at $t$ if $a$ is $r$ times differentiable at $t_1^-(t)$ and $s$ times differentiable at $t_1^+(t)$.
\end{lemma}
\begin{proof}
 According to formula \eqref{tpunkt} for their derivatives, $t_2^\pm$ are $r$ times differentiable at $t$ if $a$ is $r$ times differentiable at $t$ and $b$ is $r$ times differentiable at $t_2^\pm$, and correspondingly for $t_1^\pm$. Therefore, equations \eqref{WF} for $\ddot a$ and $\ddot b$ contain one more derivative on the left-hand side than on the right, and the claim follows.
\end{proof}

\begin{proof}[Proof of Theorem \ref{thm:+-}:]
  Choosing $V\in]\tilde V,1[$, $D\in]0,\tilde D[$, $v\in]0,\tilde v[$,
     $T>\tilde T$ and $A>\tilde A$ in the definition \eqref{Omega}
     of $\Omega$ ensures, according to Proposition \ref{prop:Asymptotik+-}, that
     $H$ defined in \eqref{H} cannot have fixed points on $\partial\Omega$.
 Obviously, $\Omega$ is a bounded open subset of Banach space $B$ defined in \eqref{B} containing the origin.
 Lemma \ref{lem:Stetigkeit} shows that $H$ is a homotopy, Lemma \ref{lem:Kompaktheit} proves its compactness.
 Furthermore, by definition, $H(0,\cdot)$ is the zero mapping and fixed points
 of $H$ satisfy \eqref{Fixpunkt}. In particular, fixed points of $H(1,\cdot)$
 satisfy \eqref{WF} with initial data \eqref{eq:initial-a} and
 \eqref{eq:initial-b}.  Thus, the existence of a solution to 
 \eqref{WF}, \eqref{eq:initial-a}-\eqref{eq:initial-b}
 follows from the
 Leray-Schauder Theorem \ref{thm:Schauder}.
 
  Concerning the regularity statement b), we observe that $\left.b\right|_{[T^-,T^+[}\in C^{n+1}$ by assumption, $b$ is once differentiable at $T^+$ by the piecewise definition of $H$ and $\left.b\right|_{]T^+,\infty[},\, \left.a\right|_{]0,\infty[}\in C^2$ by lemma \ref{lem:Regularitaet}. Consequently, also by Lemma \ref{lem:Regularitaet}, $\left.b\right|_{]T^+,\infty[}\in C^3$.
 If $n\ge1$, we can now apply the Lemma another $n$ times, alternately to $a$ at $t>0$ and $b$ at $t>T^+$, to find $\left.a\right|_{]0,\infty[\backslash\{\sigma_1,\sigma_2,\dots\}}\in C^{2+n}$ and $\left.b\right|_{]T^+,\infty[\backslash\{\tau_2,\tau_3,\dots\}}\in C^{3+n}$.
 Now, another application of the Lemma e.g. at $t\in]0,\sigma_1[$ yields no more regularity than $C^{2+n}$ because $b(t_2^-)$ is only $C^{1+n}$, but for $t\in]\sigma_1,\infty[\backslash\{\sigma_2,\sigma_3,\dots\}$, $b(t_2^-)\in C^{3+n}$, so one finds $\left.a\right|_{]\sigma_1,\infty[\backslash\{\sigma_2,\sigma_3,\dots\}}\in C^{3+n}$ and, by three more applications, $\left.b\right|_{]\tau_2,\infty[\backslash\{\tau_3,\tau_4,\dots\}}\in C^{4+n}$, $\left.a\right|_{]\sigma_1,\infty[\backslash\{\sigma_2,\sigma_3,\dots\}}\in C^{4+n}$ and $\left.b\right|_{]\tau_2,\infty[\backslash\{\tau_3,\tau_4,\dots\}}\in C^{5+n}$. Iteratively, the regularity of the segments as claimed in Theorem~\ref{thm:+-}b) is proven. The proof for $(\sigma_k)$ and $(\tau_k)$ works in the same way, the only difference being the lower initial regularity at $\sigma_1$ and $\tau_1$.
\end{proof}

\section{Proof of Theorem 1.2}

In this section we turn to the Synge equations, i.e., the FST equations without
the advanced terms.  In order to prove Theorem \ref{thm:-} we have to adopt our
strategy slightly. We define again $$\|(x,y)\|:=\max\left(\|\dot
x\|_\infty,\|\dot y\|_\infty, \sup_{t\in\mathbb R}(1+|t|)|\ddot x(t)|,
\sup_{t\in\mathbb R}(1+|t|)|\ddot y(t)|\right)$$ and make $$B^\prime:=\{(x,y)\in
C^2(\mathbb R,\mathbb R^2)\mid x(0)=\dot x(0)=y(0)=\dot
y(0)=0,\|(x,y)\|<\infty\}$$ a Banach space w.r.t. that norm.
Furthermore, we
define a map $H^\prime$ in almost the same way as we defined $H$ in
\eqref{H} in the preceding section except that this time we omit the advanced term:
\begin{equation}\label{Hstrich}
 \begin{split}
 &H_1^\prime(\lambda,x,y)(t):=-x_0(t)+a_0+\int_0^t\frac{P_1^\prime(\lambda,X,Y)(s)}{\sqrt{1+P_1^\prime(\lambda,X,Y)(s)^2}}\text ds \\
 &H_2^\prime(\lambda,x,y)(t):=-y_0(t)+b_0+\int_0^t\frac{P_2^\prime(\lambda,X,Y)(s)}{\sqrt{1+P_2^\prime(\lambda,X,Y)(s)^2}}\text ds \,,
 \end{split}
\end{equation}
where $$X:=x+x_0, Y:=y+y_0$$ with $C^2$ trajectories $(x_0, y_0)$ satisfying the Cauchy data and
\begin{equation*}
 \begin{split}
  &\|\dot x_0\|,\|\dot y_0\|<1,\\
  &\inf_{t\in\mathbb R}(x_0(t)-y_0(t))>0,\\
  &(\dot x_0-\dot y_0)(-1)<0,\,(\dot x_0-\dot y_0)(1)>0,\\
  &\ddot x_0(t)\ge0,\,\ddot y_0(t)\le0\text{ for }t\in\mathbb R\,,\\
  &\ddot x_0(t)=\ddot y_0(t)=0\text{ for }|t|\ge 1\,.
 \end{split}
\end{equation*}
The velocity and force terms are given by
\begin{align*}
 P_1^\prime(\lambda,X,Y)(t):=(1-\lambda)\frac{\dot x_0(t)}{\sqrt{1-\dot x_0(t)^2}}+\lambda\frac{\dot a_0}{\sqrt{1-\dot a_0^2}}+\lambda\int_0^t\kappa_aF_1^\prime(X,Y)(s)\text ds \\
 P_2^\prime(\lambda,X,Y)(t):=(1-\lambda)\frac{\dot y_0(t)}{\sqrt{1-\dot y_0(t)^2}}+\lambda\frac{\dot b_0}{\sqrt{1-\dot b_0^2}}+\lambda\int_0^t\kappa_bF_2^\prime(X,Y)(s)\text ds
\end{align*}
and
\begin{align*}
 F_1^\prime(X,Y)(t):=&\frac{1+\dot Y\left(t_2^-(X,Y,t)\right)}{1-\dot Y\left(t_2^-(X,Y,t)\right)}\frac1{\left(X(t)-Y(t_2^-(X,Y,t))\right)^2} \\
 F_2^\prime(X,Y)(t):=&-\frac{1-\dot X\left(t_1^-(X,Y,t)\right)}{1+\dot X\left(t_1^-(X,Y,t)\right)}\frac1{\left(Y(t)-X(t_1^-(X,Y,t))\right)^2} \,.
\end{align*}
Again, $H^\prime(0,\cdot)$ is the zero mapping whereas, if $(x,y)$ is a fixed point of $H^\prime(\lambda,\cdot)$, then
\begin{equation}\label{Fixpunkt-}
\begin{split}
 \frac{\text d}{\text dt}\left(\frac{\dot{X}(t)}{\sqrt{1-\dot{X}(t)^2}}\right)=&\frac{\ddot X(t)}{(1-\dot X(t)^2)^{\frac32}} 
 = (1-\lambda)\frac{\ddot x_0(t)}{(1-\dot x_0(t)^2)^{\frac32}}+\lambda\kappa_aF_1^\prime(X,Y)(t)\,,\\
 \frac{\text d}{\text dt}\left(\frac{\dot{Y}(t)}{\sqrt{1-\dot{Y}(t)^2}}\right)=&\frac{\ddot Y(t)}{(1-\dot Y(t)^2)^{\frac32}} 
 = (1-\lambda)\frac{\ddot y_0(t)}{(1-\dot y_0(t)^2)^{\frac32}}+\lambda\kappa_bF_2^\prime(X,Y)(t)
 \end{split}
\end{equation}
holds -- in particular, if $\lambda=1$, then $(X,Y)$ solve the
equations of motion \eqref{WF} for $\epsilon_+=0, \epsilon_-=1$ and
Cauchy data \eqref{Anfangswerte}.

This time we want to obtain global a priori bounds on possible fixed points of
$H^\prime:[0,1]\times M^\prime\rightarrow C^1(\mathbb R,\mathbb R^2)$ with
\begin{align*}
M^\prime:=\{&(x,y)\in C^1(\mathbb R,\mathbb R^2)\mid x(0)=y(0)=\dot x(0)=\dot y(0)=0,\, \|\dot X\|_\infty,\|\dot Y\|_\infty<1\,,\\
&\forall_{t\in\mathbb R}X(t)> Y(t)\}
\end{align*}
in terms of Newtonian Cauchy data only:
\begin{proposition}
 For any given Newtonian Cauchy data \eqref{Anfangswerte}, there are constants \newline
 $\tilde v^\prime,\tilde V^\prime\in]0,1[$ and $\tilde A^\prime,\tilde D^\prime,\tilde T^\prime>0$ such that, for any fixed point $(x,y)$ of $H^\prime$, 
 \begin{align*}
  &\|\dot X\|_\infty,\|\dot Y\|_\infty<\tilde V^\prime\,\,, \\
  &\inf_{t\in\mathbb R}(X-Y)(t)>\tilde D^\prime\,\,, \\
  &\sup_{t\le-\tilde T^\prime}(\dot X-\dot Y)(t)<-\tilde v^\prime\,\,, \inf_{t\ge \tilde T^\prime}(\dot X-\dot Y)(t)>\tilde v^\prime\,\,, \\
  &|\ddot X(t)|, |\ddot Y(t)|<\frac{\tilde A^\prime}{1+|t|}\,\,.
 \end{align*}
\end{proposition}

The strategy of proof stems from the following observation: Imitating the proof of Proposition \ref{prop:Asymptotik+-} in
order to get an estimate for $\dot Y$ at times $t>0$ and omitting the advanced
terms appearing there, we would need an upper bound for $\dot Y(t_2^-)$ which
in Proposition \ref{prop:Asymptotik+-} above was provided by the prescribed
trajectory strip. In contrast, we observe
that an upper bound for the now missing term $\dot Y(t_2^+)$ would already be
given by $\dot b_0$ because $\ddot Y$ is always negative. This leads to the idea
of estimating the velocity for negative times first: In this case, by reversing
the corresponding signs in the corresponding calculation in the proof of
Proposition \ref{prop:Asymptotik+-} above, the retarded term takes the
role the advanced term played in the estimate for $t>0$
above, and a lower bound for
$\dot Y(t_2^-)$ at times $t<0$ is needed, which is again given by $\dot b_0$. An
estimate for $\dot X(t)$, $t<0$, is obtained in the same way, and in a second
step, the obtained estimates provide the velocity bounds required in order to
get hands on $\dot X$, $\dot Y$ at future times.

\begin{proof}
We start the calculation by deducing from equation \eqref{Fixpunkt-}
 \begin{equation}\label{Energie-}
 \begin{split}
  &\frac{\text d}{\text dt}\left(\frac{1-\dot b_0\dot X(t)}{\sqrt{1-\dot X(t)^2}}\right)
 = \frac{(\dot X(t)-\dot b_0)\ddot X(t)}{(1-\dot X(t)^2)^{\frac32}} \\
  =& (1-\lambda)\frac{(\dot X(t)-\dot b_0)\ddot x_0(t)}{(1-\dot x_0(t)^2)^{\frac32}}
   + \lambda\kappa_a\left[\frac{(\dot X(t)-\dot b_0)(1+\dot Y(t_2^-))}{1-\dot Y(t_2^-)}\frac1{(X(t)-Y(t_2^-))^2}\right]\,.
  \end{split}
 \end{equation}
{\bf Step 1 ($t\le0$):} As in the proof of Proposition~\ref{prop:Asymptotik+-},
we want to use \eqref{tpunkt} and \eqref{Vpunkt} to estimate the right-hand side
by a total differential. Instead of equation \eqref{Ypunktkleiner0}, we now know, since $\dot Y(t)>\dot b_0$ for $t<0$, that
\begin{equation*}
 \begin{split}
 (\dot X(t)-\dot b_0)(1+\dot Y(t_2^-)) =& (1+\dot b_0)(\dot X(t)-\dot Y(t_2^-)) + (\dot Y(t_2^-)-\dot b_0)(1+\dot X(t)) \\
 \ge& (1+\dot b_0)(\dot X(t)-\dot Y(t_2^-))\,.
 \end{split}
\end{equation*}
Integration of \eqref{Energie-} from $t$ to 0 gives
\begin{align*}
 \frac{1-\dot b_0\dot a_0}{\sqrt{1-\dot a_0^2}}-\frac{1-\dot b_0\dot X(t)}{\sqrt{1-\dot X(t)^2}} \ge& (1-\lambda)\left[\frac{2\dot x_0(t)}{\sqrt{1-\dot x_0(t)^2}} - \frac{2\dot a_0}{\sqrt{1-\dot a_0^2}}\right] \\
 &+ \lambda\kappa_a\left[\frac{1+\dot b_0}{X(t)-Y(t_2^-)}-\frac{1+\dot
 b_0}{a_0-Y(t_2^-(0))}\right],
\end{align*}
and similar estimates as in the time-symmetric case lead to
\begin{equation*}
 \frac{1-|\dot b_0|}{\sqrt{1-\dot X(t)^2}}
 \le \frac4{\sqrt{1-\dot a_0^2}} + \frac{4\kappa_a}{a_0- b_0} - \lambda\kappa_a\frac{1+\dot b_0}{X(t)-Y(t_2^-)}
\end{equation*}
and
\begin{equation*}
 \sup_{t\le0}|X(t)|\le\sqrt{1-\frac{(1-|\dot b_0|)^2}{\left(\frac{4}{\sqrt{1-\dot a_0^2}}+\frac{4\kappa_a}{a_0-b_0}\right)^2}}=:V_a^- \,.
\end{equation*}
The estimate $V_b^-$ for $\sup_{t\le0}|Y(t)|$ is obtained in an analogous way.\newline

\textbf{Step 2 ($t>0$):} Now that we know bounds on the retarded velocities, we
can bound $\dot X(t)$ for $t>0$ by considering $\frac{1-V_a^-\dot
X(t)}{\sqrt{1-\dot X(t)^2}}$ and proceeding as in the proof of
Proposition~\ref{prop:Asymptotik+-} above. Again, the estimate for
$\dot Y$ for positive times is obtained in the same way, so $\tilde V^\prime$
exists as required.\newline

$\tilde v^\prime, \tilde A^\prime, \tilde D^\prime, \tilde T^\prime$ can then be
derived as in the time-symmetric case in the proof of
Proposition~\ref{prop:Asymptotik+-} above.
\end{proof}
\begin{proof}[Proof of Theorem \ref{thm:-}:]
The continuity and compactness proof in Lemma \ref{lem:Stetigkeit} and \ref{lem:Kompaktheit} for $H$ carries over to $H^\prime$ by crossing out the advanced terms, hence the existence of a fixed point of $H^\prime$ can be inferred as in the proof of Theorem \ref{thm:+-}.
\end{proof}

\section{Proof of Theorem \ref{thm:Spielmodell}}

In this section we turn to the FST toy model.
Theorem \ref{thm:Spielmodell} exemplifies that the technique used so far is in
principle also applicable of proving the existence of global solutions to
mixed-type equations, i.e., equations involving advanced as well as retarded
terms, satisfying Cauchy data. In order to treat equations 
\eqref{Spielmodell}, it is sufficient to employ the Banach space
$$B:=\{(x,y)\in C^1(\mathbb R,\mathbb R^2)\mid x(0)=\dot x(0)=y(0)=\dot y(0)=0,\|(x,y)\|<\infty\}$$
with
$$\|(x,y)\|:=\max(\|\dot x\|_\infty, \|\dot y\|_\infty)\,.$$
The second derivatives in the previous proofs came only into play for the sake
of estimating differences of the velocities appearing in the force law.
$H^{\prime\prime}$ is defined as $H^\prime$ in \eqref{Hstrich} ff., where
now
\begin{align*}
 F_1^{\prime\prime}(X,Y)(t):=&\frac1{\left(X(t)-Y(t_2^-(X,Y,t))\right)^2} + \frac1{\left(X(t)-Y(t_2^+(X,Y,t))\right)^2} \\
 F_2^{\prime\prime}(X,Y)(t):=&-\frac1{\left(Y(t)-X(t_1^-(X,Y,t))\right)^2} -\frac1{\left(Y(t)-X(t_1^+(X,Y,t))\right)^2} \,.
\end{align*}
Fixed points again satisfy \eqref{Fixpunkt-}, which for $\lambda=1$
means that they solve the equations of motion of the FST toy model
\eqref{Spielmodell} for the given Cauchy data \eqref{Anfangswerte-spiel}. Now the following a priori estimates are needed:
\begin{proposition}
 For any given Newtonian Cauchy data, there are constants $\tilde v^{\prime\prime},\tilde V^{\prime\prime}\in]0,1[$ and $\tilde D^{\prime\prime},\tilde T^{\prime\prime}>0$ such that, for any fixed point $(x,y)$ of $H^{\prime\prime}$, 
 \begin{align*}
  &\|\dot X\|_\infty,\|\dot Y\|_\infty<\tilde V^{\prime\prime}\,\,, \\
  &\inf_{t\in\mathbb R}(X-Y)(t)>\tilde D^{\prime\prime}\,\,, \\
  &\sup_{t\le-\tilde T^{\prime\prime}}(\dot X-\dot Y)(t)<-\tilde v^{\prime\prime}\,\,, \inf_{t\ge \tilde T^{\prime\prime}}(\dot X-\dot Y)(t)>\tilde v^{\prime\prime}\,\,.
 \end{align*}
\end{proposition}
\begin{proof}
We first assume that $\dot a_0-\dot b_0\le0$ (and, consequently, $\dot X(t)-\dot b_0\le0$ for $t<0$) and consider $t\le0$. Even though \eqref{Spielmodell} is not Lorentz invariant, we again compute the ``kinetic energy'' in the Lorentz frame boosted by $b_0$:
\begin{equation}\label{EnergieS}
\begin{split}
 &\frac{\text d}{\text dt}\left(\frac{1-\dot b_0\dot X(t)}{\sqrt{1-\dot X(t)^2}}\right)
 =\frac{(\dot X(t)-\dot b_0)\ddot X(t)}{(1-\dot X(t)^2)^{\frac32}} \\
 =& (1-\lambda)\frac{(\dot X(t)-\dot b_0)\ddot x_0(t)}{(1-\dot x_0(t)^2)^{\frac32}}
 +\lambda \frac{\dot X(t)-\dot b_0}{\left(X(t)-Y(t_2^-)\right)^2} 
 + \lambda\frac{\dot X(t)-\dot b_0}{\left(X(t)-Y(t_2^+)\right)^2}\,.
 \end{split}
\end{equation}
Since the acceleration of $Y$ is never positive, the closest possible route to
$X$ that $Y$ can take is $\tilde Y(t):=b_0+\dot b_0t$, and, using Lemma
\ref{lem:t+-},
$$X(t)-Y(t_2^\pm)\ge X(t)-\tilde Y(\tilde t_2^\pm)\ge \frac12(X(t)-\tilde Y(t))$$
holds.
Equation \eqref{EnergieS} then reduces to
\begin{align*}
 \frac{\text d}{\text dt}\left(\frac{1-\dot b_0\dot X(t)}{\sqrt{1-\dot X(t)^2}}\right)
 \ge& \frac{-2\ddot x_0(t)}{(1-\dot x_0(t)^2)^{\frac32}}
 +4\lambda \frac{\dot X(t)-\dot b_0}{\left(X(t)-b_0-\dot b_0t)\right)^2} \,;
\end{align*}
by integration from $t<0$ to 0 one arrives at
\begin{align*}
 \frac{1-\dot b_0\dot X(t)}{\sqrt{1-\dot X(t)^2}}
 \le& \frac{1-\dot b_0\dot a_0}{\sqrt{1-\dot a_0^2}} +\frac{2\dot a_0}{\sqrt{1-\dot a_0^2}} - \frac{2\dot x_0(t)}{\sqrt{1-\dot x_0(t)^2}}+\frac{4\lambda}{a_0-b_0}-\frac{4\lambda}{X(t)-b_0-\dot b_0t} \\
 \le&\frac4{\sqrt{1-\dot a_0^2}}+\frac4{a_0-b_0} \,.
\end{align*}
From this, one finds the bound
\begin{equation*}
 \sup_{t\le0}|\dot X(t)|\le\sqrt{1-\frac{(1-|\dot b_0|)^2}{\left(\frac4{\sqrt{1-\dot a_0^2}}+\frac4{a_0-b_0}\right)^2}}=:V_a^-\,.
\end{equation*}
The estimate $V_b^-$ for $|\dot Y(t)|$ at negative times is inferred in the same way. For positive times, we look at
\begin{align*}
 &\frac{\text d}{\text dt}\left(\frac{1-V_b^-\dot X(t)}{\sqrt{1-\dot X(t)^2}}\right)
 =\frac{(\dot X(t)-V_b^-)\ddot X(t)}{(1-\dot X(t)^2)^{\frac32}} \\
 \le& (1-\lambda)\frac{(\dot X(t)-V_b^-)\ddot x_0(t)}{(1-\dot x_0(t)^2)^{\frac32}}
 +\lambda \frac{\dot X(t)-\dot Y(t_2^-)}{\left(X(t)-Y(t_2^-)\right)^2} + \lambda\frac{\dot X(t)-\dot Y(t_2^+)}{\left(X(t)-Y(t_2^+)\right)^2}
\end{align*}
and use
$$\dot X(t)-\dot Y(t_2^\pm)\le2\frac{\dot X(t)-\dot Y(t_2^\pm)}{1\pm\dot Y(t_2^\pm)}\,.$$
We can proceed with the help of equation~\eqref{Vpunkt} as in the proof of
Proposition~\ref{prop:Asymptotik+-}. The same procedure can be applied to get
the corresponding bound for $\dot Y$. If, contrary to our assumption, $\dot
a_0>\dot b_0$, we have to reverse the order of our proof, i.e., do the first
part for positive and the second part for negative times.\newline

The remaining inequalities are then proved analogous to the arguments in Proposition~\ref{prop:Asymptotik+-}.
\end{proof}
\begin{proof}[Proof of Theorem~\ref{thm:Spielmodell}:]
The proof that $H^{\prime\prime}$ is continuous and compact is similar to the
one of Lemmata~\ref{lem:Stetigkeit} and \ref{lem:Kompaktheit}, with considerable
simplifications due to the absence of the velocity terms $\sigma$ and $\rho$ introduced in
\eqref{eq:vel-factors}. Again, the existence of a
fixed point of $H^{\prime\prime}$ can be inferred as in the proof of Theorem
\ref{thm:+-}.
\end{proof}

\section{Proof of Theorem \ref{thm:Rekonstruktion}}

In this last section we prove that only finites strips of the solutions are
needed to identify them uniquely.
\begin{figure}
    \includegraphics[trim=3.6cm 8.6cm 3.6cm 2.7cm, clip, width=0.8\textwidth]{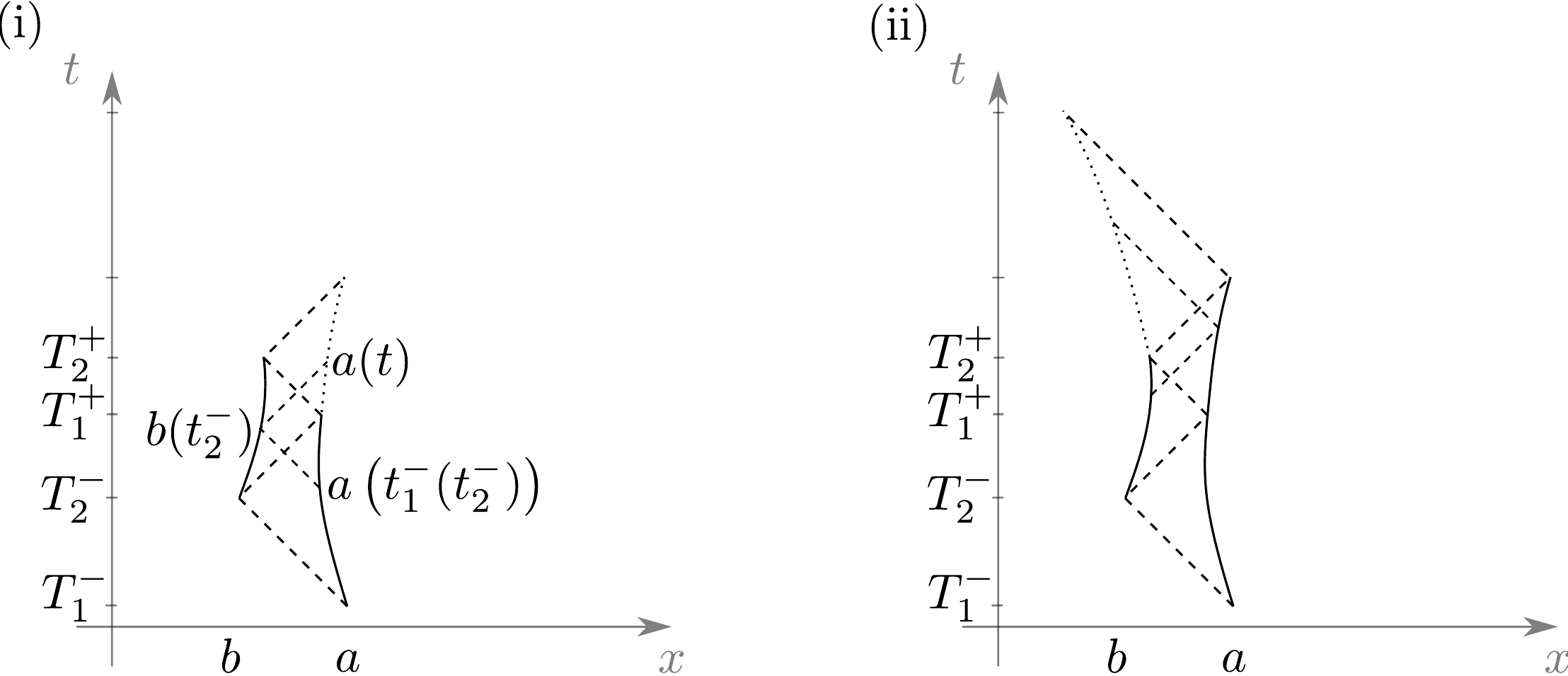}
    \caption[...]{\label{fig:Rekonstruktion}Reconstruction of solutions from the segments $\left.a\right|_{[T_1^-,T_1^+]}$ and $\left.b\right|_{[T_2^-,T_2^+]}$.
     (i) Step 1: Reconstruction of $\left.a\right|_{[T_1^+,t_1^+(T_2^+)]}$ (the
     dotted part), e.g., $(t,a(t))$ from $b(t_2^-)$ and $a(t_1^-(t_2^-))$ and
     their derivatives;
     (ii) Step 2: Reconstruction of the dotted part
     $\left.b\right|_{[T_2^+,t_2^+(t_1^+(T_2^+))]}$.}
\end{figure}
\begin{proof}[Proof of Theorem~\ref{thm:Rekonstruktion}:]
The idea is to solve equation~\eqref{WF} for $b$ for the advanced term in order
to compute $a$ up to time $t_1^+(T_2^+)$ and to iterate this procedure as
indicated in Figure~\ref{fig:Rekonstruktion}.
To reconstruct $a(t)$ for $t\in[T_1^+, t_1^+(T_2^+)]$, we
rewrite equation~\eqref{WF} for $\dot b$ as
\begin{equation*}
 \dot a(t_1^+)=f\left(-2\left(b(t)-a(t_1^+)\right)^2\left[\frac{\ddot b(t)}{\kappa_b(1-\dot b(t)^2)^{\frac32}}
 +\frac12\frac{1-\dot a(t_1^-)}{1+\dot a(t_1^-)}\frac1{(b(t)-a(t_1^-))^2}\right]\right)
\end{equation*}
or
\begin{equation*}\label{Fortsetzung}
 \dot a(t)=f\left(-2\left(b(t_2^-)-a(t)\right)^2\left[\frac{\ddot b(t_2^-)}{\kappa_b(1-\dot b(t_2^-)^2)^{\frac32}}
 +\frac12\frac{1-\dot a(t_1^-(t_2^-))}{1+\dot a(t_1^-(t_2^-))}\frac1{(b(t_2^-)-a(t_1^-(t_2^-)))^2}\right]\right)
\end{equation*}
with $f(u)=\frac{u-1}{u+1}$.
Trajectory $b$ at the retarded times as well as $a$ at the double
retarded times being the given input for the reconstruction, this equation can
be interpreted as an ordinary differential equation
\begin{equation}\label{GDG}
\dot a(t)=f\left(g(t,a(t))\right)
\end{equation}
for $a$ with given initial value $a(T_1^+)$, where
\begin{align*}
g(t,x):=-2\left(b(t_2^-)-x\right)^2\left[\frac{\ddot b(t_2^-)}{\kappa_b(1-\dot b(t_2^-)^2)^{\frac32}}
 +\frac12\frac{1-\dot a(t_1^-(t_2^-))}{1+\dot a(t_1^-(t_2^-))}\frac1{(b(t_2^-)-a(t_1^-(t_2^-)))^2}\right]
 \end{align*}
 and each term $t_2^-$ also depends on $t$ and $x$ according to definition~\eqref{t+-}.
The map $g$ is well-defined for all $(t,x)$ such
that
$t_2^-$ exists, i.e., on
$$D_g:=\left\{(t,x)\mid\exists s\in[T_2^-,T_2^+]:t-s=|x-b(s)|\right\}\,,$$
the set of all space-time points which can be reached from the initial
trajectory segment with speed of light. The velocity field $f\circ g$ is
thus well-defined on
$$D_{f\circ g}:=\{(t,x)\in D_g\mid g(t,x)\neq-1\}\,.$$ 
By definition of the initial segments, we have 
$T_1^+-T_2^-=a(T_1^+)-b(T_2^-)$, which implies that $(T_1^+,a(T_1^+))$ is
contained in $D_{g}$. That this point also lies in $D_{f\circ g}$ follows from the fact
that the equation of motion \eqref{WF} for $b$ is satisfied at time $T_2^-$.
The strips $\left.a\right|_{[T_1^-,T_1^+]}$ and $\left.b\right|_{[T_2^-,T_2^+]}$ are $C^\infty$ because the equations of motion have one more derivative on the left-hand side than on the right.
 Therefore, $f\circ g$ is differentiable w.r.t. $x$, in particular locally Lipschitz continuous, and \eqref{GDG} has a unique solution with a certain maximal lifetime.
 Since \eqref{GDG} is just the reordered equation of motion, $a$ itself solves
 it as long as $(t,a(t))$ remains in $D_{f\circ g}$. As shown in Proposition~\ref{prop:Asymptotik+-},
 we have
 \begin{align}
     \label{eq:global-bounds}
     |\dot a|\leq C<1
     \qquad
     \text{and}
     \qquad
    \inf_{t\in\mathbb R}|a(t)-b(t)|=:D>0.
 \end{align}
 Therefore, such a maximal lifetime $T_{\max}$ exists and
 $T_{\max}-T_2^+=|a(T_{\max})-b(T_2^+)|\ge D$, i.e.,
 $T_{\max}=t_1^+(a,b,T_2^+)$. 
 Consequently,
 $\left.a\right|_{[T_1^+,t_1^+(a,b,T_2^+)]}$ is the unique maximal solution to
 \eqref{GDG}.  In an analogous way,
 $\left.b\right|_{[T_2^+,t_2^+(t_1^+(a,b,T_2^+))]}$ can be uniquely
 reconstructed and, iteratively, the whole solution in the future because 
 the bounds \eqref{eq:global-bounds} are uniform in time. The
 reconstruction of the solution in the past is performed in the
 same way, solving \eqref{WF} for
 the retarded instead of advanced terms.
\end{proof}
\begin{remark}
Note that, if this construction is performed with arbitrarily prescribed
(smooth) trajectory strips (not a priori gained from a global solution), it can
end after a finite lifetime. More precisely, this is always the case if the
prescribed trajectories or the ones obtained after finite iterations exhibit
accelerations that require an attractive or zero force between the particles,
reach the speed of light in finite time or approach a light line as an
asymptote. Moreover, the obtained trajectories need not be differentiable at
times $T_1^+$, $T_2^+$, $t_1^+(T_2^+)$, etc. Apart from very particular cases, it
is unknown which reasonable additional conditions on the trajectory strips
ensure global existence of this construction.
The corresponding uniqueness assertion for the toy model
$\eqref{Spielmodell}$ in three spacial dimensions
can be found in \cite{DirkNicola}. There, it is also discussed how the initial
strips can be restricted in order to increase the regularity at times $T_1^+$,
$T_2^+$, $t_1^+(T_2^+)$, etc. 
\end{remark}

\vskip1cm

\paragraph{Acknowledgment} The authors thank G.\ Bauer and D. Dürr for valuable
discussions on FST electrodynamics and express their gratitude for
the hospitality of G.\ Bauer's 
group at FH Münster. Furthermore, D.-A.D. thanks H.-O. Walther and R. Nussbaum for
discussions on delay equations. This work was partially funded by the Elite Network of
Bavaria through the Junior Research Group ``Interaction between Light and
Matter''.

\vskip1cm

\bibliographystyle{plain}

\vskip1cm

\end{document}